\documentclass[12pt, reqno]{amsart}

\usepackage[a4paper,margin=1in]{geometry}

\usepackage{amsmath,amssymb,amsfonts,amsthm}
\usepackage{mathrsfs}
\usepackage{bm}

\usepackage{enumerate}
\usepackage{hyperref}
\usepackage{appendix}
\hypersetup{
    colorlinks=true,
    linkcolor=blue,
    citecolor=blue,
    urlcolor=blue
}

\newtheorem{theorem}{Theorem}[section]
\newtheorem{lemma}[theorem]{Lemma}
\newtheorem{corollary}[theorem]{Corollary}
\newtheorem{proposition}[theorem]{Proposition}

\theoremstyle{remark}
\newtheorem{remark}[theorem]{Remark}
\newtheorem{definition}{Definition}

\newcommand{\rank}{\operatorname{rank}}

\newcounter{localreduction}
\renewcommand{\thelocalreduction}{\Roman{localreduction}}

\newenvironment{localreduction}[1][]
  {\refstepcounter{localreduction}%
   \par\medskip
   \noindent\textbf{(\thelocalreduction)}\ifx&#1&\else\ \textit{#1}\fi
   \par\smallskip}
  {\par\medskip}

\title{Submanifolds of higher rank with curvature normals of constant length}

\author{Santiago Castañeda-Montoya}
\address{Universidad de Antioquia, Colombia}
\email{santiago.castanedam@udea.edu.co}

\author{Guillermo Lobos}
\address{Departamento de Matemática, Universidade Federal de São Carlos, Brazil}
\email{lobos@ufscar.br}
  
\author{Carlos Olmos}
\address{Ciem-CONICET, ciudad universitaria, Córdoba, Argentina}
\email{carlos.olmos@unc.edu.ar}
\thanks{The first author was supported by the Universidad de
Antioquia. The second author was supported by UFSCar, and partially
supported by the FAPESP-CEPID Project 2024/00923-6 and FAPESP Project
2022/16097-2. The third author was supported by the visiting program
of PPGM, UFSCar, and partially supported by CIEM-CONICET}

\date{\today}


\begin{document}

\begin{abstract}
We study the local geometry of Euclidean submanifolds of rank at
least two whose curvature normals with respect to the flat part of
the normal bundle have constant length, equivalently, whose adapted
third fundamental form has constant eigenvalues. The complete case
was fully settled by Di Scala and the third author, who showed that
such a complete submanifold has constant principal curvatures. Around
a regular point, we show that $M$ is generated, via parallel
manifolds, by a hypersurface $N$ whose extrinsic factors are
contained in spheres and, when of dimension at least two, have rank
one; conversely, such factors can be assembled to construct examples
of $M$. The proof relies on the isoparametric rank theorem to rule
out non-umbilical factors. As a further application, when the inner
products of the curvature normals are constant, we show that there
are exactly two eigendistributions, one of which is one-dimensional
and autoparallel, extending, for arbitrary codimension when the normal bundle is
flat, curvature-homogeneity results of Tsukada and the recent ones of Bryant, Florit,
and Ziller for hypersurfaces.
\end{abstract}

\maketitle

\section{Introduction}\label{intro}

The purpose of this paper is to study the local geometry of Euclidean
submanifolds of rank at least two, i.e., with the flat part of the normal
bundle of dimension at least two, whose curvature normals with respect
to the flat part $\nu_0M$ of the normal bundle have constant length. Equivalently, the third fundamental form with respect to $\nu _0M$ has constant eigenvalues. 

The complete case was settled by Di Scala and Olmos in \cite{DO}, where
they proved that a complete submanifold satisfying these conditions has
constant principal curvatures. They also provided a local counterexample
showing that the conclusion does not hold without completeness. Here we
describe the local structure in this setting.

Around a regular point, we obtain a distinguished one-dimensional
autoparallel eigendistribution whose orthogonal complement is integrable.
A local integral manifold $N$ of this orthogonal complement, together
with its parallel manifolds, gives a local description of $M$.
Moreover, $N$ inherits strong restrictions on its curvature normals and
extrinsic factors; in particular, it is contained in a sphere, and every
extrinsic factor of dimension at least two has rank one.

Our results are related to the rank theorem for
submanifolds: the homogeneous version \cite{O1} and its non-homogeneous
version \cite{DO}. The rank theorem was crucial in \cite{O2}.

A further consequence is obtained when the inner products of the
curvature normals are constant. In this case, if the principal
curvatures are not locally constant, then the number $g$ of distinct
curvature normals equals $2$, and one of the corresponding
eigendistributions is one-dimensional and autoparallel.

\smallskip

The third fundamental form is closely related to the Ricci tensor
through the formula expressing the latter in terms of this form and
the shape operator in the direction of the mean curvature vector. The
condition that the second fundamental form be algebraically constant
can be regarded as the extrinsic analogue of curvature homogeneity.

Curvature homogeneity was studied by Tsukada \cite{Ts} for hypersurfaces from an intrinsic perspective. In particular, he showed that, in Euclidean space, and in the sphere for hypersurfaces of dimension at least four, a curvature-homogeneous hypersurface which is neither isoparametric nor an extrinsic product must have constant sectional curvature. The remaining three-dimensional spherical case was recently settled by Bryant, Florit, and Ziller \cite{BFZ}, who found additional non-isoparametric examples. Their results are more general in that they are formulated intrinsically, whereas ours are extrinsic but allow arbitrary codimension, under the assumption that the normal bundle is flat. In this sense, the two approaches are complementary.

\medskip

We now state the main results, which summarize the content of the
paper.

\smallskip

\begin{theorem}\label{thm:main1}
Let $M^n$ be a locally irreducible and full submanifold of the Euclidean
space $\mathbb R^{n+p}$ whose curvature normals $\eta_1,\dots,\eta_g$
with respect to the flat part $\nu_0M$ of the normal bundle have constant
length, where $g\geq2$; their associated eigendistributions are denoted
by $E_1,\dots,E_g$. Assume that $M$ does not have constant principal
curvatures around any of its points. Then, near any regular point $p$, $M$ is
described as follows: let $\eta_k$ be a curvature normal of maximal
length among the non-parallel curvature normals. Then
\begin{enumerate}
    \item $E_k$ is an autoparallel distribution and $\dim E_k=1$
    (we denote by $X$ a unit vector field spanning $E_k$ and by
    $\gamma_x(t)$ the integral curve of $X$ starting at $x$).

    \item $E_k^\perp$ is integrable. We denote by $N$ a fixed local
    integral manifold of $E_k^\perp$ through $p$; such an $N$ is called a \textbf{generating hypersurface} of $M$.

    \item For any fixed small $t_0$, $x\mapsto\gamma_x(t_0)-x$, $x\in N$,
    is a parallel normal field of $N$, denoted by $\xi^{t_0}$. Moreover,
    the parallel manifold $N_{\xi^{t_0}}$ coincides (locally) with the
    integral manifold of $E_k^\perp$ through $\gamma_p(t_0)$.
   \item $M$ is the union of parallel manifolds. Namely, 
   \[
M=
\bigcup_{t\in(-\epsilon,\epsilon)}N_{\xi^t}
\qquad\text{(locally)}.
\]

    \item $\nu_0N=\mathbb RX_{|N}\oplus(\nu_0M)_{|N}$.

    \item The eigendistributions of $N$ with respect to $\nu_0N$ are
    $\bar E_i=E_{i|N}$, and $(\eta_i)_{|N}$ are the orthogonal
    projections onto $(\nu_0M)_{|N}$ of the curvature normals associated
    with $\bar E_i$, $i=1,\dots,g$, $i\neq k$. 
    \item The curvature normals of $N$ with respect to $(\nu_0M)_{|N}$ have constant length. 
  \item $\eta_{k|N}$ is a parallel non-zero normal isoparametric section of $N$.
  \item $N$ is contained in a sphere. 
 \item If $N'$ is an extrinsic factor of $N$, then
$TN'=E_{i|N}$, for some $i\in\{1,\dots,g\}\setminus\{k\}$. 
Moreover, this is the only eigendistribution associated to $\nu _0N'$.
\end{enumerate}
\end{theorem}

\smallskip

We now give a more precise description of the extrinsic
factors of the generating hypersurface $N$.

\begin{theorem}\label{thm:main1B} Let $N$ be  a generating hypersurface of a locally  irreducible Euclidean submanifold $M$, with non-constant principal curvatures  and $\nu _0M$-curvature normals of constant length. Then any of  the local extrinsic factors $N'$ of $N$, regarded as full submanifolds of its affine span $\mathbb A'$, is contained in a sphere and is of one of the following types: 
\begin{enumerate}[(A)] 
    \item $\dim N'\geq 2$, and $\nu _0N'= \mathbb R \vec{p}$, where $\vec{p}$ is the position vector field with respect to the centre of the sphere in $\mathbb A'$ that contains $N'$ (in particular, $\rank N'=1$ and the curvature normal is 
    $-\frac{1}{\vert \vec p\vert ^2 }\vec p$). 
     \item $N'$ is an \emph{admissible}  one dimensional submanifold.
\end{enumerate}    
\end{theorem}

This theorem and Theorem~\ref{thm:main1}(5) imply that, if $M$ has in
addition a flat normal bundle, then any extrinsic factor of $N$ of
dimension at least $2$ is a sphere. 

For the concept of admissible one-dimensional submanifolds see Definition \ref{def:6}. Any one-dimensional submanifold, contained in a sphere,  with curvature normal of constant length is admissible. 

In Section~\ref{sec:examples}, we prove that if each factor of a
submanifold $N$ is either as in (A) of the previous theorem or is a
one-dimensional submanifold of a sphere with curvature normal of
constant length, then $N$ is the generating hypersurface of a locally
irreducible submanifold $M$ with $\nu_0M$-curvature normals of
constant length. It remains open whether the same conclusion holds
when a one-dimensional factor is merely admissible, without its
curvature normal having constant length.

\medskip

We now briefly indicate some of the ideas in the proof  of the above theorems. The set $\{1,\dots,g\}\setminus\{k\}$ is
the union of two disjoint subsets, $I_0$ and $I_1$, where $I_0$
consists of the uncoupled indices, that is, those for which
$\eta_i\perp\eta_k$, and $I_1$ consists of the coupled indices, that
is, those for which
\(
\langle\eta_i,\eta_k\rangle=\langle\eta_i,\eta_i\rangle
\).
This, in particular, implies that $\eta_{k|N}$ is a parallel normal
isoparametric section of $N$. Using the isoparametric rank theorem,
together with the property that no three curvature normals of $M$ can
lie on a line, we can exclude the factors of $N$ where $\eta_k$ is
non-umbilical, so that the remaining factors have constant principal
curvatures and higher rank (see Lemma~\ref{lem:line} and
Corollary~\ref{cor:exFactor}).

We always have $I_1\neq\emptyset $. Moreover, if $I_0\neq\emptyset $,
then $|I_1|\geq2$ (see Section~\ref{sec:long}, Lemma~\ref{lem:I_0}, and Remark~\ref{rem:820}). If $i,j\in I_1$, then $\langle\eta_i,\eta_j\rangle$ is not constant by Lemma~\ref{lem:non<,>}. 

From Theorem~\ref{thm:main1} and the preceding
facts, we obtain the following result.

\begin{theorem}\label{thm:main2}
Let $M^n$, $n\geq2$, be a locally irreducible and full submanifold of the
Euclidean space $\mathbb R^{n+p}$ of rank at least $2$. Let
$\eta_1,\dots,\eta_g$ be its curvature normals with respect to the flat
part $\nu_0M$ of the normal bundle. Assume that $M$ does not have
constant principal curvatures around any of its points and that the
inner products $\langle\eta_i,\eta_j\rangle$ of any two curvature
normals are constant. Then $g=2$, and one of the eigendistributions has
dimension $1$ and is autoparallel.
\end{theorem}

\begin{corollary}\label{cor:main3}
Let $M^n $, $n\geq 2$,  be an irreducible and full submanifold of Euclidean space with
flat normal bundle. Assume that the second fundamental form is algebraically constant and that $M$ is not an isoparametric submanifold. 
Then $M$ has exactly two eigendistributions, one of them of dimension $1$ and autoparallel, and  has positive constant sectional curvature. 
\end{corollary}

\medskip

The paper is organized as follows. In Section~\ref{sec:prel}, we recall the basic
facts and notation concerning curvature normals and eigendistributions.
Section~\ref{sec:structure} is devoted to the local structure of submanifolds with curvature
normals of constant length. In Section~\ref{rec}, we describe how $M$ is
recovered from an integral manifold $N$ of the distinguished
eigendistribution, while Section~\ref{sec:long} establishes further restrictions
involving the longest non-parallel curvature normal and the extrinsic
factors of $N$. In Section~\ref{sec:inner}, we study the inner products of the
curvature normals and derive consequences from their constancy. Section~\ref{sec:proof} contains the proofs of the main results, and Section~\ref{sec:examples} presents a general construction of examples.
Finally, the Appendix~\ref{App:A} collects some basic facts about submanifolds with
constant principal curvatures and higher rank. These facts do not seem
to have been treated in the literature in the setting considered here,
in particular because there is no Weyl group associated with the flat
part of the normal bundle.

\section{Preliminaries and basic facts}\label{sec:prel}
Let $M^n \subset \mathbb{R}^{n+p}$ ($n \geq 2$) be a full local submanifold 
that is locally irreducible around every point. We may assume that M is contractible; hence every vector bundle over M is trivial. As usual, $A$ denotes the shape operator and $\alpha$ the second fundamental form.
Let $\nu_0M$ denote the flat part of the normal bundle of $M$ that we assume has constant dimension, since we work locally. Let 
$\tilde{\nu}_0M$ be a non-zero parallel subbundle of $\nu_0M$.\footnote{Although we will mainly be concerned with the case
$\tilde\nu_0M=\nu_0M$, we will need this more general setting for the
hypersurfaces $N$ of $M$ introduced later.} 
By the Ricci identity, all shape operators of sections of $\tilde{\nu}_0M$ 
commute, and so there exists, at each point $p \in M$, an orthogonal decomposition
\[
T_pM = E_1(p) \oplus \dots \oplus E_g(p)
\]
with respect to which all shape operators associated with vectors in 
$(\tilde{\nu}_0M)_p$ are simultaneously diagonalized.

Associated with this decomposition, there are curvature normals
\[
\eta_i(p) \in (\tilde{\nu}_0M)_p, \qquad i = 1, \dots, g(p),
\]
such that for every $\xi \in (\tilde{\nu}_0M)_p$
\begin{equation}\label{eq:00}
A_\xi\big|_{E_i(p)}
=
\lambda_i(\xi)\operatorname{Id}
=
\langle \xi,\eta_i(p)\rangle\operatorname{Id}.
\end{equation}

On an open dense subset  of $M$, $g(p)$ and the dimensions of
$E_1(p), \dots, E_{g(p)}(p)$ are locally constant. Since we are working 
locally, we may assume that $g := g(p)$ and $\dim E_1, \dots, \dim E_g$ 
are constant on $M$. Thus, $E_1, \dots, E_g$ define smooth distributions, 
called the $\tilde{\nu}_0 M$-eigendistributions. Moreover, $\eta_1, \dots, \eta_g$ 
are smooth sections of $\tilde{\nu}_0 M$, called the curvature normals 
with respect to $\tilde{\nu}_0 M$.

Let $\xi$ be an arbitrary section of $\tilde{\nu}_0 M$. Observe that
\[
R^\perp_{X,Y} \xi = 0.
\]
Hence, by the Ricci identity, the shape operator $A_\xi$ commutes with 
all shape operators associated with arbitrary sections of $\nu M$. 
Therefore, each eigendistribution is invariant under every shape operator of $M$. Equivalently, 
\begin{equation}\label{eq:-1}
\alpha (E_i,E_j)= 0  \text{ if } i\neq j.
\end{equation}
\medskip

\textbf{General Assumptions.} From now on, we assume that \(g\geq 2\) and that the length of each curvature normal \(a_i: =\|\eta_i\|\), with respect to \(\tilde{\nu}_0M\), is constant (for short, \(M\) has \(\tilde{\nu}_0M\)-curvature normals of constant length).

\medskip

In order to make the formulation of our main theorems precise, we
introduce the following definition.

\medskip

\begin{definition}\label{def:AA}
A Euclidean (local) submanifold is said to have
\emph{$\tilde\nu_0$-curvature normals of constant length} (or curvature
normals  with respect to $\tilde\nu_0$ of constant length) if:
\begin{enumerate}
    \item The flat part $\nu_0M$ of the normal bundle $\nu M$ has constant
    dimension and hence is smooth.

    \item $\tilde\nu_0$ is a nonzero parallel (and hence flat) subbundle
    of $\nu_0M$.

    \item The eigendistributions $E_1,\dots,E_g$ associated with
    $\tilde\nu_0M$ have constant dimension and hence are smooth, where
    $g\geq1$ is constant.

    \item Each of the associated curvature normal fields $\eta_i$ has
    constant length $\|\eta_i\|$, $i=1,\dots,g$.
\end{enumerate}
\end{definition}
\begin{definition}\label{def:AA2}
A Euclidean submanifold is said to be
\begin{itemize}

    \item[--]  \emph{locally irreducible} if it is irreducible in a neighborhood
    of each of its points; 

    \item[--] \emph{locally full} if it is full in a neighborhood of each of
    its points.
\end{itemize}
\end{definition}

\bigskip

The fact that $M$ has $\tilde \nu_0 M$-curvature normals of constant length 
 is equivalent to the fact that the tensor
\[
B_p^0 = \sum_{i=1}^r A^2_{\xi_i(p)}
\]
has constant eigenvalues, where $\xi_1(p), \dots, \xi_r(p)$ is an orthonormal 
basis of $(\tilde{\nu}_0 M)_p$.
\medskip

Such a tensor is called the third fundamental form of $M$ with respect to 
$\tilde{\nu}_0 M$.

We assume, without loss of generality, that 
\begin{equation}\label{eq:not09}
\|\eta_i\| \geq \|\eta_j\| \quad \text{if } i < j, \quad i, j \in \{1, \dots, g\}.
\end{equation}



Let $\hat I \subset \{1, \dots, g\}$ be defined by
\begin{equation}\label{eq:393}
    \hat I = \bigl\{\, i \in \{1, \dots, g\} \colon \nabla^\perp \eta_i = 0 \,\bigr\}.
\end{equation}

\begin{remark}\label{rem:123}
If $\hat I = \{1, \dots, g\}$, then for any parallel section $\xi$ of 
$\tilde{\nu}_0 M$, the shape operator $A_\xi$ has constant eigenvalues 
$\langle \eta_1, \xi \rangle, \dots, \langle \eta_g, \xi \rangle$. 
Thus, $\xi$ is a non-trivial isoparametric parallel section. It follows 
from Theorem~4.5.10 of~\cite{BCO} that $M$ is contained in a sphere, 
say $S^{N-1}$.
If $\xi$ is generic and $g \geq 2$, then the eigenvalues 
$\langle \eta_1, \xi \rangle, \dots, \langle \eta_g, \xi \rangle$ are 
pairwise distinct. Hence, $\xi$ is a parallel isoparametric normal section which is  not  umbilical. Then, by Theorem~\ref{thm:iso-rank}, 
 $M$  is either an 
isoparametric hypersurface of the sphere or an orbit of an $s$-representation. 
\end{remark}
\begin{remark}\label{rem:232}
Assume that $g = 1$. Then $E_1 = TM$ is an autoparallel distribution 
and thus, by Lemma~\ref{lem:2.1}~(2), $\nabla^\perp \eta_1 = 0$ 
($n = \dim M \geq 2$). If $\eta_1 = 0$, then $M$ is not full, since 
any non-trivial parallel section of $\tilde{\nu}_0 M$ has null shape 
operator. If $\eta_1 \neq 0$, then the map $q \mapsto q + \|\eta_1\|^{-2}\eta_1$ 
is constant and hence $M$ is contained in a sphere of radius $\|\eta_1\|^{-1}$.
\end{remark}

\medskip

By Remark~\ref{rem:123}, we may assume that $\hat I$ is a proper, possibly 
empty, subset of $\{1, \dots, g\}$. Since we are working locally, we may 
assume, for every $i \in \{1, \dots, g\} \setminus \hat I$, that the curvature 
normal $\eta_i$ is nowhere parallel on $M$.
Let
\begin{equation}\label{eq:A}
k := \min\bigl(\{1, \dots, g\} \setminus \hat I\bigr).
\end{equation}

In~\cite{DO}, the whole bundle $\nu_0 M$ was considered, rather than a 
parallel subbundle $\tilde{\nu}_0 M$ of it. However, most of the local results 
obtained there remain valid for smaller parallel subbundles. 

Let us recall the following 
result, which does not require the curvature normals to have constant length.
\begin{lemma}[see~\cite{DO}, Lemma~2.1]\label{lem:2.1}
In the general notation of this section:
\begin{enumerate}
    \item $E_i$ is autoparallel if and only if $\nabla^\perp_Z \eta_i = 0$ 
    for all $Z \in E_i^\perp$.
    \item If $\dim E_i \geq 2$, then $E_i$ is autoparallel if and only if 
    $\eta_i$ is a parallel normal section.
\end{enumerate}
\end{lemma}
\medskip

Let $h_{i,j}: M\to \mathbb R$, $i, j\in \{1, \dots , g\}$, $i\neq j$, be defined by 
\[ h_{i,j} = \langle \eta _j, \eta _j\rangle -\langle \eta _i , \eta _j\rangle = \langle \eta _j-\eta _i, \eta _j\rangle\]
(see Section 4 of \cite{DO}). Then, from the Cauchy-Schwarz inequality,  
\begin{equation}\label{eq:82}
    h_{i,j}> 0 \text{ if } j<i.
    \end{equation} 

\begin{definition} \label{def:1}
A point $p \in M$ is called \emph{generic} if there exists an open 
neighborhood $\bar \Omega$ of $p$ in $M$ such that for all $i \in \{1, \dots, g\}$ 
either $h_{i,k}$ is identically zero on $\bar \Omega$ or 
$h_{i,k}$ is nowhere vanishing in $\bar\Omega_p$.
\end{definition}
\medskip

\noindent It is standard to show that the set $\bar \Omega$ of generic points is open and dense in $M$. 
\medskip

Observe, by (\ref{eq:82}), that $h_{i,k}$ is nowhere vanishing on $M$ if $i>k$. Moreover, by \eqref{eq:A}, if $i<k$, then $\eta_i$ is parallel. 

\begin{lemma}[see~\cite{DO}, Lemma~4.1; \cite{O1}, p. 611]\label{lem:4.1}
Under the notation and assumptions of this section:
\begin{enumerate}
    \item $E_k$ is autoparallel and $\dim E_k=1$.
    \item $E_k^\perp$ is an integrable distribution of $M$.
\end{enumerate}
\end{lemma}
\begin{proof}
Let $p\in M$ be a generic point and let $\bar \Omega _p$ be as in Definition \ref{def:1}, and let 
 $$J:=\{i\in\{1,\dots,g\}: h_{i,k}\equiv0 \text{ on }\bar\Omega_p\}.$$ Observe that $J\subset \{1, \dots , k-1\}$ and thus $\nabla ^\perp\eta _i = 0$ if $i\in J$ (see the paragraphs below Definition \ref{def:1}). Observe that $\lambda _k(\eta _i)$ is constant, for all $i\in J$, since $\langle \eta _k, \eta _k\rangle$ is constant by assumption and $\langle \eta _k, \eta _k\rangle =\langle\eta _i,\eta _k\rangle$, since $h_{i,k}=0$.

One has that 
$\hat \nu: = \text {span}\{\eta _i:i\in J\}$
defines a parallel subbundle of $\tilde \nu _0M$ and 
\begin{equation}\label{eq:3s}
\lambda _k(\zeta) \text{ is constant for any parallel section } \zeta \text{ of }  \hat \nu.
\end{equation}

Let $\xi$ be a parallel normal field of $M$ with $\xi_p = (\eta_k)_p$. By the Cauchy-Schwarz inequality, since $\xi$ and $\eta _k$ have constant lengths, 
\begin{equation}\label{eq:4s}
    \lambda _k(\xi) \text{ attains a maximum at } p.
\end{equation}

On the one hand, $\lambda_k(\xi_p) = \lambda_i(\xi_p)$ if and only if 
$h_{i,k}(p)=0$, or equivalently, if and only if $i\in J$. On the other hand,  if  $i\in J$, then $\|\eta _i\|\geq \|\eta _k\|$ and hence $$\lambda _i((\eta _i)_p)>\lambda _k((\eta_i)_p)$$
by the Cauchy-Schwarz inequality.  It is standard to show that there exists  $z\in \hat \nu _p$ such that $(\eta _k)_p + z = \xi _p + z$ distinguishes all eigenvalue functions $\lambda _1, \dots , \lambda _g$. Let $\zeta$ be a parallel section of $\hat \nu$ such that $\zeta _p = z$. Then, by \eqref{eq:3s} and \eqref{eq:4s},

\noindent(a) $\lambda _k(\xi +\zeta)$ attains its maximum at $p$ and thus,  $\mathrm{d}_p(\lambda _k(\xi +\zeta)) = 0$

\noindent (b) $\xi +\zeta $ distinguishes all eigenvalue functions near $p$.
\medskip

Let $$T^p : = A_{\xi + \zeta} - \lambda _k(\xi + \zeta)\operatorname{Id}.$$
Then $T^p$ satisfies the Codazzi identity at $p$ and $\ker T^p = E_k$ near $p$. This implies that $E_k$ is an autoparallel distribution at $p$ (see above mentioned references). 
Since the set $\bar \Omega$ of generic  points is dense, we conclude that 
$E_k$ is autoparallel on $M$.  Since $\eta_k$ is nowhere parallel, Lemma \ref{lem:2.1} implies part~(1).

Let $X$ and $Y$ be tangent vector fields on $M$ lying in $E_k^\perp$. By part~(1),
$$
\nabla^\perp_X\nabla^\perp_Y\eta_k
=
0
=
\nabla^\perp_Y\nabla^\perp_X\eta_k.
$$
Since $\eta_k$ is a section of the flat bundle $\tilde\nu_0M$,
$R^\perp_{X,Y}\eta_k=0$.
Hence,
$$
\nabla^\perp_{[X,Y]}\eta_k=0.
$$
Since \(\eta_k\) is nowhere parallel, it follows from the above equality and Lemma~\ref{lem:2.1}(1) that \(E_k^\perp\) is involutive. Therefore, \(E_k^\perp\) is parallel.
\end{proof}

\begin{lemma}\label{lem:4.1+} 
Under the notation and assumptions of this section, let \(N\) be an arbitrary integral manifold of \(E_k^\perp\) and \(X\) a unit vector field tangent to \(E_k\). Denote also by \(X\) its restriction to \(N\), and let \(\bar A\) be the shape operator of \(N\) as a hypersurface of \(M\) with respect to the unit normal field \(X\). Then 
\begin{enumerate}
\item $X$ is a parallel normal field of $N$, regarded as a submanifold of the ambient Euclidean space. 
\item $A_X= \bar A$.
\item \(E_i\) is invariant under \(A_X\), and \((A_X)_{|E_i}\) is a scalar multiple $\mu _i$ of the identity ($i\neq k$). 
\item $TN$ is invariant under all shape operators of $M$.
\end{enumerate}
\end{lemma}
\begin{proof}
The first two assertions follow easily from the fact that \(X\) is a parallel normal field of \(N\), regarded as a hypersurface of \(M\), and that \(\alpha(E_k,TN)=0\). The proof of (3) is standard, using the Codazzi identity; see, for instance, the proof of Lemma~3.5(3) in \cite{O1} or that of Lemma~4.11 in \cite{DO}. Part (4) follows from \eqref{eq:-1}, since \(TN\) is the direct sum of the restrictions to \(N\) of all eigendistributions except \(E_k\).
\end{proof}

Let us denote by
\begin{equation}
\label{eq:notation1}
\tilde{\nu}_0N:=\mathbb{R}X_{|N}\oplus (\tilde{\nu}_0M)_{|N}.
\end{equation}
This is a parallel and flat subbundle of the normal bundle of \(N\).

\medskip

\begin{corollary}\label{cor:11}
The different eigendistributions of \(N\) determined by the parallel subbundle
\(\tilde{\nu}_0N\subset\nu_0N\) coincide with the restrictions to \(N\) of the
eigendistributions, different from \(\mathbb R X\), associated to
\(\tilde{\nu}_0M\). Moreover, the corresponding curvature normals $\tilde \eta _i$ are given by 
\begin{equation}\label{eq:rst2}
\tilde \eta_i:=\mu_iX_{|N} + (\eta_i)_{|N},
\qquad i=1,\ldots,g,\; i\neq k.
\end{equation}\qed 
\end{corollary}

We will need the following general result later. Since we have not found a
proof in the literature, we include one.

\begin{lemma}\label{lem:aux9}
Let $S$ be a Euclidean submanifold and let $\hat\nu _0 S$ be a parallel and
flat subbundle of $\nu_0S$. Assume that the different eigendistributions
$E_1,\dots,E_g$ determined by $\hat\nu_0 S$ have constant dimension and are
therefore smooth ($g\geq 2$). Assume that each eigendistribution is autoparallel and
that $E_1^\perp=E_2\oplus\dots\oplus E_g$ is integrable. Then $E_1^\perp$
is autoparallel, and $S$ is locally a product of submanifolds
$S=S_1\times S_2$, where $S_1$ is an integral manifold of $E_1$ and
$S_2$ is an integral manifold of $E_1^\perp$.
\end{lemma}
\begin{proof}
Let $\eta_1,\dots,\eta_g$ be the curvature normals associated with
$E_1,\dots,E_g$. Let $q\in S$ and let $\xi$ be a parallel section of
$\hat\nu_0S$ such that the functions
\[
\lambda_i:=\langle\eta_i,\xi\rangle,\qquad i=1,\dots,g,
\]
are pairwise distinct at every point in a neighborhood of $q$. Note that
the shape operator satisfies
\[
A_\xi|_{E_i}=\lambda_i\operatorname{Id}_{E_i}.
\]
Let $i,j\geq2$, and let $X_i$ and $Y_j$ be tangent vector fields of $S$
that lie in $E_i$ and $E_j$, respectively. If $i=j$, then
$\nabla_{X_i}Y_j$ lies in $E_i$ and hence in $E_1^\perp$ by assumption.
Let us assume that $i\neq j$ and let $Z_1$ be a tangent vector field that
lies in $E_1$. Then
\[
\langle (\nabla_{X_i} A)_\xi Y_j, Z_1\rangle
= \langle(\lambda_i-\lambda_1)\nabla_{X_i}Y_j,Z_1\rangle.
\]
By the Codazzi identity, and using the fact that
$\langle [X_i,Y_j],Z_1\rangle=0$, since $E_1^\perp$ is integrable, we
obtain
\begin{eqnarray*}
0
&=&\left\langle
(\lambda_i-\lambda_1)\nabla_{X_i}Y_j
-(\lambda_j-\lambda_1)\nabla_{Y_j}X_i,Z_1
\right\rangle\\
&=&\left\langle
(\lambda_i-\lambda_1)\nabla_{X_i}Y_j
-(\lambda_j-\lambda_1)\nabla_{Y_j}X_i
-(\lambda_j-\lambda_1)[X_i,Y_j],Z_1
\right\rangle\\
&=&\left\langle
(\lambda_i-\lambda_j)\nabla_{X_i}Y_j,Z_1
\right\rangle.
\end{eqnarray*}
Then
$\langle\nabla_{X_i}Y_j,Z_1\rangle=0$, and therefore $E_1^\perp$ is
autoparallel. Since $E_1$ is also autoparallel, both distributions are
parallel. The lemma follows from Moore's lemma, since
$\alpha(E_1,E_1^\perp)=0$.
\end{proof}

\begin{remark}\label{rem:equalL}
Let $S$ be a Euclidean submanifold and let $\hat\nu _0 S$ be a parallel and
flat subbundle of $\nu_0S$. Assume that the different eigendistributions
$E_1,\dots,E_g$ determined by $\hat\nu_0 S$ have constant dimension and are
therefore smooth (we assume that $g\geq 2$; see Remark~\ref{rem:equalL2}).  Let  us assume that all the associated curvature normals $\eta _1, \dots , \eta _g$ have the same constant length. Then any eigendistribution is autoparallel. In fact, let $i\in \{1, \dots , g\}$, $g\geq 2$. If $\eta _i$ is a parallel normal section, then $E_i$ is autoparallel by Lemma \ref{lem:2.1}(1). If $\eta _i$ is not parallel we can choose it as a longest non-parallel curvature normal and Lemma \ref{lem:4.1}(1) implies that $E_i$ is autoparallel. Thus, for every $i\in 
\{1, \dots ,g\}$, $E_i$ is autoparallel. Moreover, if some $\eta _i$ is not parallel, then $E_i^\perp$ is integrable and hence $S$ splits by Lemma \ref{lem:aux9}. Then, if $S$ is locally irreducible, any curvature normal is parallel. In this case, any generic parallel section $\hat \nu _0S$ is a parallel and non-umbilical isoparametric section. Hence $S$ has constant principal curvatures by Theorem \ref{thm:iso-rank}.
\end{remark}

\medskip

\begin{remark}\label{rem:equalL2}
Let $S$ be a full Euclidean submanifold of dimension at least $2$, and
let $\hat\nu_0 S$ be a parallel and flat subbundle of $\nu_0S$. Assume
that the different eigendistributions $E_1,\dots,E_g$ determined by
$\hat\nu_0 S$ have constant dimension and are therefore smooth. If
$\dim\hat\nu_0 S\geq 2$, then $g\geq 2$. In fact, if $g=1$, the same argument as in the case of a flat normal
bundle, using the Codazzi identity, shows that $\eta_1$ is parallel.
Then a non-trivial parallel section $\xi$ of $\hat\nu_0S$ perpendicular
to $\eta_1$ satisfies $A_\xi=0$. Hence, $S$ reduces codimension, a
contradiction. 
\end{remark}

\section{The structure of submanifolds with curvature normals of constant length}\label{sec:structure}
Let $M^n\subset \mathbb R^{n+p}$ ($n\geq 2$) be a full local submanifold that is locally irreducible at every point. Let $\tilde{\nu}_0M$ be a parallel and flat subbundle of $\nu_0M$, and assume that the $\tilde{\nu}_0M$-curvature normals have constant length. We keep the notation and general assumptions of Section~\ref{sec:prel}.

Denote by
$$
S(x)\qquad (x\in M)
$$
the integral manifold of $E_k^\perp$ through $x$.
With this notation, any fixed integral manifold $N$ of $E_k^\perp$ is of the form
$N=S(p)$ for some $p\in M$. Once such an integral manifold has been fixed, we shall simply write $N$ as in the previous section.

The integral curve of $X$ starting at $q$ will be denoted by $\gamma _q(t)$. 
 If $\phi_t$ denotes the (local) flow of $X$, then
\begin{equation}\label{eq:flow}
\gamma_q(t)=\phi_t(q).
\end{equation}

Since we are working locally, we may assume that all these integral curves, starting at every $q\in N$, are defined on $(-\varepsilon,\varepsilon)$.
 Moreover, since $E_k$ is autoparallel and $X$ is a unit vector field tangent to $E_k$, each $\gamma_q$ is a unit-speed geodesic of $M$ (see Lemmas~\ref{lem:4.1} and~\ref{lem:4.1+}).
\smallskip

For any fixed $t\in(-\varepsilon,\varepsilon)$, define a vector field $\xi^t$ of $\mathbb R^{n+p}$ along $N$ by
\begin{equation}\label{eq:campo}
\xi^t_q :=\phi_t(q)-q=\gamma_q(t)-q.
\end{equation}

\begin{lemma}\label{lem:4.2} Let $t_0\in(-\varepsilon,\varepsilon)$ be arbitrary. Then 
\begin{enumerate}
    \item $\xi^{t_0}$ defines a parallel normal field of $N$. 
    \item $S(\gamma _q(t_0))$ does not depend on $q\in N$.
    \item  $S(\gamma _q(t_0))$ is a parallel manifold to $N$ in $\mathbb R^{n+p} $. Namely, $$N_{\xi ^{t_0}}= S(\gamma _q(t_0)).$$
\end{enumerate}
\end{lemma}

\begin{proof}
Since $E_k$ is autoparallel, $E_k^\perp$ is parallel along $E_k$ in $M$. Moreover, since $\alpha(E_k,E_k^\perp)=0$, it follows that $E_k^\perp$ is constant along the integral curves of $E_k$ in $\mathbb R^{n+p} $. Therefore, for every $q\in N$ and every $t\in(-\varepsilon,\varepsilon)$,
\begin{equation}\label{eq:proof1}
T_{\gamma_q(t)}S(\gamma_q(t))=T_qN,
\end{equation}
when both tangent spaces are regarded as subspaces of $\mathbb R^{n+p} $. Consequently,
\begin{equation}\label{eq:proof2}
\nu_{\gamma_q(t)}S(\gamma_q(t))=\nu_qN,
\end{equation}
where the normal spaces are taken in $\mathbb R^{n+p} $.

For any fixed $q\in N$, let us consider the curve of $\mathbb R^{n+p} $ $t\mapsto h_q(t):=\xi ^t_q$. Then 
\begin{eqnarray*}\frac{\,\mathrm d}{\mathrm d t}h_q(t) &=& \frac{\,\mathrm d}{\mathrm d t}(\phi _t(q) -q) = X_{\phi _t(q)}
\\ &=& X_{\gamma _q(t)}\in 
\nu_{\gamma_q(t)}S(\gamma_q(t))=\nu _qN.
\end{eqnarray*}
Since $h_q(0)=0$ and $\nu_qN$ is a vector subspace of $\mathbb R^{n+p} $, it follows that
$
h_q(t)\in \nu_qN
$
for all $t\in(-\varepsilon,\varepsilon)$. Therefore,
$
\xi^{t_0}_q\in \nu_qN
$
for every $t_0\in(-\varepsilon,\varepsilon)$. This shows that $\xi^{t_0}$ is a normal field of $N$.

Let $q\in N$, $v\in T_qN$ be arbitrary and let $c(s)$ be a smooth curve in $N$ such that $c(0)=q, \, c'(0)=v$. Then  $\gamma _{c(s)}$ is a smooth variation of unit-speed geodesics of $M$. Hence, 
the Jacobi field along $\gamma _{q}(t)$,  $$J(t):=\tfrac{\, \partial}{\partial s}|_{0}\, \gamma _{c(s)}(t)$$ is everywhere perpendicular to 
$\gamma' _q(t)= X_{\gamma _q(t)}$. Therefore, by \eqref{eq:proof1}, 
$$J(t) \in T_{\gamma _q(t)}S(\gamma _q(t)) = T_qN.$$
If $t_0\in (-\varepsilon , \varepsilon )$, then 
\begin{eqnarray*}
  \frac{\,\mathrm d}{\mathrm d s}|_{0}\, \xi ^{t_0}_{c(s)} &=& \frac{\,\mathrm d}{\mathrm d s}|_{0}(\gamma _{c(s)}(t_0)-c(s)) 
  =
  \frac{\,\mathrm d}{\mathrm d s}|_{0}\, \gamma _{c(s)}(t_0) -v \\ &=& J(t_0) -v \in T_qN.
\end{eqnarray*}
Then $\nabla ^\perp _v\xi ^{t_0} =0$, which finishes the proof of (1).

Let $q,q'\in N$ and let $c:[0,1]\to N$ be a smooth curve with $c(0)=q$ and $c(1)=q'$. Consider the variation of unit-speed geodesics $\gamma_{c(s)}$. By the same arguments as in the proof of~(1), for any $s_0\in [0,1]$, the Jacobi field along $\gamma_{c(s_0)}(t)$,
$$
J^{s_0}(t):=\tfrac{\, \partial}{\partial s}|_{s_0}\, \gamma_{c(s)}(t),
$$
is everywhere perpendicular to $\gamma'_{c(s_0)}(t)$. In particular,
$$
J^{s_0}(t_0)\in (E_k^\perp)_{\gamma_{c(s_0)}(t_0)}
$$
for every $t_0\in (-\varepsilon,\varepsilon)$.  The curve 
$$s\mapsto \gamma _{c(s)}(t_0)=\phi _{t_0}(c(s))$$ joins $\gamma _q(t_0)$ with $\gamma _{q'}(t_0)$ and its velocity coincides with 
$J^s(t_0)\in (E_k^\perp)_{\gamma _{c(s)}(t_0)}$. Then 
\begin{equation}\label{eq:igualdad4}
S (\phi _{t_0}(q)) = S(\gamma _q(t_0)) = S(\gamma _{q'}(t_0)) = S (\phi _{t_0}(q')),
\end{equation}
where $q$ and $q'$ are arbitrary elements of $N$.
This implies that $\phi _{t_0}: N\to S (\phi _{t_0}(q))$ is a diffeomorphism with inverse $\phi _{-t_0}$.

Observe from \eqref{eq:campo} that the parallel map 
\begin{equation}\label{eq:parmap}
    x\mapsto x + \xi ^{t_0}_x \qquad x\in N,
    \end{equation}
    coincides with the restriction to $N$ of  $\phi _{t_0}$. Therefore, $N_{\xi ^{t_0}} = S(\phi _{t_0}(q))$.
\end{proof}

\subsection{Recovering M from N} \label{rec} 

We keep the notation and general assumptions of the previous sections.  Since we work locally, we may assume that the map, from
$(-\varepsilon,\varepsilon)\times N$ into $M$, 
$
(t,q)\mapsto \phi_t(q)= \gamma _q(t)$
is a diffeomorphism. 
\medskip

From now on, we assume that $$\tilde \nu_0M=\nu_0M.$$
In this case,
\begin{equation}\label{eq:u8e}
\nu_0N=\mathbb RX\oplus\hat\nu_0N
\end{equation}
where $\hat \nu _0N = \nu _0M|N$.
 The above equality follows from the vanishing of
$R^\perp_{E_k,E_k^\perp}$, the fact that $E_k$ is one-dimensional,
and the integrability of $E_k^\perp$ (see, for instance,
Lemma~5.1.11 of \cite{BCO}). 

\begin{lemma}\label{lem:n_k} \ 

\begin{enumerate}
    \item $\alpha(X,X)$ is a section of $\nu _0M$.
    \item $\alpha(X,X) = \eta _k$.
    \end{enumerate}
\end{lemma}
\begin{proof}
Let $p\in M$ be arbitrary and let $N=S(p)$. By Lemma~\ref{lem:4.2},
$\xi^t$ defines a parallel normal field on $N$ for each fixed $t$
sufficiently close to $0$. Hence
$\xi_p^t=\gamma_p(t)-p\in(\nu_0N)_p$ for all $t$ near $0$. Therefore,
$$
\tfrac{\mathrm d^2\,}{\mathrm d t^2}\big|_{0}(\gamma_p(t)-p)
=
\tfrac{\mathrm d^2\,}{\mathrm d t^2}\big|_{0}\gamma_p(t)
=
\alpha(X_p,X_p)\in(\nu_0N)_p.
$$

Since $\gamma_p'(t)$ has unit length,
$\tfrac{\mathrm d^2\,}{\mathrm d t^2}\big|_{0}\gamma_p(t)\perp\gamma_p'(0)=X_p$.
Therefore,
$$
\alpha(X_p,X_p)\in(\nu_0M)_p
$$ which implies~(1).
Let $\xi\in(\nu_0M)_p$ be arbitrary. Then
$$
\langle \xi,\alpha(X_p,X_p)\rangle
=
\langle A_\xi X_p,X_p\rangle
=
\langle \langle \xi,(\eta _k)_p\rangle X_p,X_p\rangle
=
\langle \xi,(\eta_k)_p\rangle
$$
which proves~(2).
\end{proof}

In the notation of Corollary~\ref{cor:11}, let
$$
\tilde\eta_i=\mu_iX+\eta_i
$$
denote the $\nu_0N$-curvature normals of $N$, where
$i=1,\dots,g$ and $i\neq k$ (here, as before, $X$ stands for $X_{|N}$,
and $\eta_i$ stands for $(\eta_i)_{|N}$).
(Note, in the notation of Corollary~\ref{cor:11}, that 
$\nu_0N = \tilde\nu_0N$).

 Observe that we can only conclude that the curvature normals
$\eta_i$ with respect to $\hat\nu_0N$ have constant length. In general,
the curvature normals $\tilde\eta_i$ associated with $\nu_0N$ need not
have constant length.
\medskip

One has a one-parameter family $\xi^t$ of parallel normal fields on $N$.  Namely,
$$
\xi^t_q=\gamma_q(t)-q,
\qquad -\varepsilon<t<\varepsilon.
$$
Recall that, for each fixed $t$, the parallel manifold $N_{\xi^t}$ in $\mathbb R^{n+p} $ is given by
$$
N_{\xi^t}=S(\gamma_q(t)),
$$
where $q$ is any point of $N$ (see Lemma~\ref{lem:4.2}).

It follows from the tube formula that the $(\nu_0N_{\xi^t})$-curvature normals
$\tilde\eta_i^t$ of $N_{\xi^t}\subset \mathbb R^{n+p} $ are given by
\begin{equation}\label{eq:cur77}
 (\tilde \eta _i^t)_{\gamma _q(t)} := \tfrac{1}{1- \langle  (\tilde \eta _i)_q, \xi ^t_q \rangle} \, (\tilde \eta _i)_q 
 = \tfrac{1}{1- \langle  (\tilde \eta _i)_q, \gamma _q(t) -q \rangle} \, (\tilde \eta _i)_q \qquad (i\neq k).
\end{equation}

As subspaces of $\mathbb R^{n+p} $ one has that  
\begin{equation}\label{eq:cur78}(\nu _0N)_q = (\nu _0N_{\xi ^t})_{q+\xi ^t_q}. 
\end{equation}
Nevertheless, the direct summands in the decomposition
\begin{equation}\label{eq:cur79}
(\nu_0N_{\xi^t})_{q+\xi_q^t}
=
\mathbb RX_{q+\xi_q^t}
\oplus
(\nu_0M)_{q+\xi_q^t}
\end{equation}
vary with $t$, unless $X$ is a constant vector field. In this case, $X$ is a parallel tangent vector field on $M$ that is invariant under the family of shape operators. Hence, $M$ locally splits off a line, contradicting the assumption that $M$ is locally irreducible.
\medskip

Note that
\begin{equation}\label{eq:80}
\pi_2(\tilde\eta_i^t)=(\eta_i)_{|N_{\xi^t}},
\end{equation}
where $\pi_2$ denotes the projection onto the second direct summand. 

Observe that  
\begin{eqnarray}\label{eq:81}\pi _2 ((\tilde\eta _i^t)_{\gamma _q (t)}) &=&
(\tilde\eta _i^t)_{\gamma _q (t)} - \langle (\tilde\eta _i^t)_{\gamma _q (t)}, 
X_{\gamma _q(t)}\rangle X_{\gamma _q(t)}\nonumber \\ 
&=& \tfrac{1}{1- \langle  (\tilde \eta _i)_q, \gamma _q(t) -q \rangle} \, (\tilde \eta _i)_q \nonumber\\
&\ & -\langle \tfrac{1}{1- \langle  (\tilde \eta _i)_q, \gamma _q(t) -q \rangle} \, (\tilde \eta _i)_q , \gamma '_q(t)\rangle \gamma '_q(t) \nonumber \\ 
&=& \tfrac{1}{1- \langle  (\tilde \eta _i)_q, \gamma _q(t) -q \rangle} \, (\tilde \eta _i)_q \nonumber \\
&\ & -\tfrac{1}{1- \langle  (\tilde \eta _i)_q, \gamma _q(t) -q \rangle} \langle \, (\tilde \eta _i)_q , \gamma '_q(t)\rangle \gamma '_q(t).
\end{eqnarray}
Set 
\begin{equation}\label{def:001}
f_{q,i}(t): = 1- \langle  (\tilde \eta _i)_q, \gamma _q(t) -q \rangle
\end{equation}
which implies that 
\begin{equation}\label{eq:002}
f'_{q,i}(t) = - \langle  (\tilde \eta _i)_q, \gamma '_q(t) \rangle
\end{equation}

\medskip

Then \eqref{eq:81} can be written as 
\begin{eqnarray}\label{eq:84}
(\eta _i)_{\gamma _q(t)} &=& \pi _2 ((\tilde\eta _i^t)_{\gamma _q (t)}) \nonumber \\
&=& \tfrac{1}{f_{q,i}(t)} \, (\tilde \eta _i)_q  
+ \tfrac{1}{f_{q,i}(t)} f'_{q,i}(t) \gamma '_q(t).
\end{eqnarray}

The condition that 
\begin{equation}\label{eq:a_i^2}
a_i^2:=\langle (\eta _i) , (\eta _i)\rangle 
\end{equation}
is constant, and in particular along $\gamma _q(t)$, can be written from  \eqref{eq:84} as 
\begin{equation}\label{eq:85}
 \tfrac{1}{f^2_{q,i}(t)} \langle  (\tilde \eta _i)_q, (\tilde \eta _i)_q\rangle   
- \tfrac{1}{f^2_{q,i}(t)} (f'_{q,i}(t))^2 = a_i^2
\end{equation}
Equivalently, 
\begin{equation}\label{eq:86}
a_i^2f_{q,i}^2 + (f'_{q,i})^2 = \langle  (\tilde \eta _i)_q, (\tilde \eta _i)_q\rangle 
= \mu_i^2(q) + a_i^2
\end{equation}
(see Corollary~\ref{cor:11}). 
While we may assume that $a_i\geq 0$, $\mu _i$ may be negative.

A smooth solution to the  differential equation \eqref{eq:86} is equivalent, by differentiating, to a smooth solution 
of
\begin{equation}\label{eq:dereq}
\left\{
\begin{aligned}
f'_{q,i}(a_i^2f_{q,i} + f''_{q,i} ) &= 0,\\
f_{q,i}(0) &= 1,\\
f'_{q,i}(0)  &= -\mu _i(q).
\end{aligned}
\right.
\end{equation}

It is standard to show that a smooth solution $f_{q,i}$ of \eqref{eq:dereq} satisfies either the equation 
$a_i^2f_{q,i} + f''_{q,i} =0$ or is constantly 1. In the latter case, $\mu_i=0$ by (\ref{eq:002}). By explicitly solving the first equation, one obtains if $a_i\neq 0$ that 
\begin{equation}\label{eq:expsol}
     f_{q,i}(t)\ = \cos(a_it) - \tfrac{\mu _i(q)}{a_i}\sin (a_it).
\end{equation}
If $a_i=0$, then $\eta_i=0$. In our notation \eqref{eq:not09} we must have $i=g$. 
Hence, by \eqref{eq:cur79}, for every $q\in M$

\begin{equation}
    (\tilde \eta _g^t)_{\gamma _q(t)} = \tfrac{1}{1- \langle  (\tilde \eta _g)_q, \xi ^t_q \rangle} \, (\tilde \eta _g)_q  
    = \lambda(t)X_{\gamma _q(t)} = \lambda (t)\gamma _q'(t)
\end{equation}
Since $\pi_2((\tilde\eta_g^t)_{\gamma_q(t)})=0$.
Then, if $(\tilde\eta_g)_q\neq 0$, $\gamma_q'(t)$ is a scalar multiple of the
constant vector $(\tilde\eta_g)_q$. Hence,
$$
0=
\tfrac{\mathrm d^2\,}{\mathrm d t^2}\big|_{0}\gamma_p(t)
=
\alpha(X_p,X_p)
=
(\eta _k)_p, 
$$
where the last equality follows from Lemma~\ref{lem:n_k}(2). Since $\eta _k$ has constant length, $\eta _k =0$. This is a contradiction. 
Then, if $a_g=0$, $\tilde\eta_g$ must vanish identically on $N$. Hence, $\mu _g=0$. 
Then $f_{q,g} (0)=1, \,f'_{q,g}(0) = 0$, and $f''_{q,g} = 0$. This implies that $f_{q,g} = 1$ whenever 
$a_g=0$. Thus, we have proved that either $f_{q,i}(t)$ is given by \eqref{eq:expsol}, with $a_i>0$, 
or it is constantly equal to $1$. The latter case occurs if
$\gamma_q(t)-q$ is always perpendicular to $(\tilde\eta_i)_q$.

\begin{remark}\label{rem:800}
Note that $\eta_i=0$ if and only if $a_i=0$. Thus, in our convention, $i=g\neq k$
(the inequality follows from the fact that $\eta_k$ is not parallel). Moreover,
$\eta_g=0$ if and only if $\tilde{\eta}_g=0$ (see the preceding paragraphs).
\end{remark}

\begin{remark}\label{rem:820}
Assume that $f_{q,i}(t)\equiv 1$, for all $q\in N$ and all $i\neq k$. Observe that 
\begin{equation*}\label{eq:cur77h}
 (\tilde \eta _i^t)_{\gamma _q(t)} = \tfrac{1}{1- \langle  (\tilde \eta _i)_q, \gamma _q(t) -q \rangle} \, (\tilde \eta _i)_q = \tfrac{1}{f_{q,i}(t)} \, (\tilde \eta _i)_q = (\tilde \eta _i)_q.
\end{equation*}
Since $\gamma _q(t) -q =\xi ^t_q\in \nu _0N = \nu _0S(\gamma _q(t)),$ 
$\gamma '_q(t)\in \nu _0S(\gamma _q(t))$. Taking into account that 
\begin{equation*}\label{eq:cur77g}-f'_{q,i}(t) = \langle (\tilde \eta _i)_q, \gamma '(t)\rangle =0 =  \langle (\tilde \eta _i^t)_{\gamma _q(t)}, \gamma '_q(t)\rangle
\end{equation*}
one concludes that $A^t_{X_{\gamma _q(t)}} =0$, where $A^t$ is the shape operator of 
the hypersurface $S(\gamma _q(t))$ of $M$. Then $S(\gamma _q(t))$ is totally geodesic. Therefore, $E_k$ and $E_k^\perp$ are two perpendicular autoparallel
distributions. Hence, both must be parallel distributions on $M$. Since
$\alpha(E_k,E_k^\perp)=0$, it follows that $M$ locally splits by Moore's lemma (see \cite{BCO}),
contradicting the assumption that $M$ is locally irreducible at every
point (see Corollary~1.7 of \cite{BCO}). Thus, not all the functions $f_{q,i}(t)$ can be identically equal to $1$.
\end{remark}

\bigskip

\begin{localreduction}[A local reduction]\label{localreduction} 
In what follows, since we work locally, we may assume, after restricting
to a neighborhood of a generic point and perhaps replacing $N$ by a nearby
parallel manifold $\phi_{t_0}(N)$, that for each index $i\neq k$, either all
the functions $f_{q,i}(t)$, $q\in N$, are identically equal to $1$, or all
are given by \eqref{eq:expsol}. Namely, for all $q\in N$,
\begin{equation}\label{eq:repeat}
f_{q,i}(t)=\cos(a_it)-\frac{\mu_i(q)}{a_i}\sin(a_it), \quad a_i>0.
\end{equation}

\smallskip

\begin{definition}\label{def:regular}
    The points of $M$ at which the above local reduction holds are called
\emph{regular}. The set $\Omega$ of regular points is open and dense in $M$.
\end{definition}
\end{localreduction}

\medskip

\subsection{Inner products with the longest non-parallel curvature normal and related constraints} \label{sec:long}

We keep the notation and  assumptions of the previous section.
Let us denote  
\begin{equation}\label{def:5}
\left\{
\begin{aligned}
I &:= \{1,\dots,g\}\setminus\{k\},\\
I_0 &:= \left\{i\in I:
f_{q,i}(t)\equiv 1,\text{ for all }q\in N\right\},\\
I_1 &:= I\setminus I_0.
\end{aligned}
\right.
\end{equation}

\medskip

\begin{definition}\label{def:coupled}
    The curvature normal field $\eta _i$ is said to be \emph{coupled} with $\eta _k$ if $i\in I_1$, and \emph{uncoupled} if $i\in I_0$.
\end{definition}

\medskip

The above  definition is motivated by the following result:  

\begin{lemma}\label{lem:acople}
Under the general assumptions and notation of this section, let $i\in I$.
\begin{enumerate}
    \item If $i\in I_0$, then $\langle \eta_i,\eta_k\rangle=0$.
    \item If $i\in I_1$, then $\langle \eta_i,\eta_k\rangle= \langle \eta_i,\eta_i\rangle= a_i^2\neq 0$.
    \item If $i,j\in I_1$, with $i\neq j$, then $\eta _i$ is not a scalar multiple of $\eta _j$.
\end{enumerate}
\end{lemma}
\begin{proof}
 Let $x=\phi_{t}q =\gamma _q(t)$, where $q\in N$, and let $i\in I_0$. Then 
\begin{eqnarray*}
0&=&-f''_{q,i}(t) = \langle \tilde \eta _i, \gamma ''_q(t)\rangle = \langle   (\tilde \eta^t _i)_{\gamma _q(t)}, (\eta _k)_{\gamma _q(t)}\rangle \\
&=& \langle   (\tilde \eta^t _i)_x, (\eta _k)_x\rangle = \langle (\eta _i)_x, (\eta _k)_x\rangle
\end{eqnarray*}
by  Lemma \ref{lem:n_k}, equation \eqref{eq:cur77}, and the fact that  $\eta _k\perp X$ (the second derivative is in the ambient Euclidean space). This proves (1).

Assume that $i\in I_1$. By \eqref{eq:84}
\begin{equation*}
(\eta _i)_x = \tfrac{1}{f_{q,i}(t)} \, (\tilde \eta _i)_q  
+ \tfrac{1}{f_{q,i}(t)} f'_{q,i}(t) \gamma '_q(t).
\end{equation*}
(observe that $f_{q,i}(t)\neq 0$, since the parallel manifold $N_{\xi ^t}$ is non-singular).
By Lemma \ref{lem:n_k}, $(\eta _k)_x= \gamma _q''(t)$, since $X_{\gamma _q(t)} = \gamma_q'(t)$. Note that 
$\gamma_q'(t)\perp \gamma _q''(t)$, since $\gamma '(t)$ has constant length. Then 
\begin{eqnarray*}\langle (\eta _i)_x, (\eta _k)_x\rangle &=& \tfrac{1}{f_{q,i}(t)} \langle (\tilde \eta _i)_q, \gamma _q''(t)\rangle = 
-\tfrac{f_{q,i}''}{f_{q,i}(t)}\\
&=& a_i^2
\end{eqnarray*}
since $f_{q,i}$ is given by \eqref{eq:expsol}. This proves (2). If $\eta_j=c\,\eta_i$, then (2) implies that $c^2=c\neq 0$. Hence,
$c=1$ and therefore $\eta_i=\eta_j$, which implies that $i=j$, a
contradiction. Thus, (3) follows.
\end{proof}

\medskip

\begin{lemma} \label{lem:line} We keep the notation and assumptions of this section. Assume that, for some $q\in N$, 
$(\tilde \eta _i)_q$, $(\tilde \eta _j)_q$, and $(\tilde \eta _l)_q$ lie on a line, where $i,j,l\in I$ are all different. Then $i,j,l\in I_0$.
\end{lemma}
\begin{proof}
Assume that one of the indices, say $l$, does not belong to $I_0$. Equivalently, by (\ref{def:5}) and (\ref{eq:repeat}),
\begin{equation*}
f_{q,l}(t)=\cos(a_lt)-\frac{\mu_l(q)}{a_l}\sin(a_lt).
\end{equation*}

Since the three curvature normals are all different and lie on a line, there
exist $a,b\in\mathbb R$, with $a+b=1$ and $a\neq 1\neq b$, such that
\begin{equation}\label{eq:lemma7a}
(\tilde\eta_l)_q = a(\tilde\eta_i)_q+b(\tilde\eta_j)_q.
\end{equation}
This implies that
\begin{equation}\label{eq:90x}
f_{q,l}=af_{q,i}+bf_{q,j}.
\end{equation}

If $f_{q,i}=1$, then
\begin{equation}\label{eq:90y}
f_{q,l}(t)= a + bf_{q,j}(t)
\end{equation}
Hence, $f_{q,j}$ cannot be constantly $1$ and hence 
\begin{equation*}
f_{q,j}(t)=\cos(a_jt)-\frac{\mu_j(q)}{a_j}\sin(a_jt).
\end{equation*}
By taking even parts of the functions in (\ref{eq:90y}),  we obtain 
$$\cos(a_lt) = a + b\cos(a_jt).$$
Observe that the period of the left hand function is $\frac{2\pi}{a_l}$, while the period of the second hand function is $\frac{2\pi}{a_j}$. Then $a_l=a_j$ which implies that 
$$(1-b)\cos(a_lt) = a.$$ A contradiction.
 We also obtain a contradiction if $f_{q,j}=1$. Therefore,
\begin{eqnarray*}
\cos(a_lt)-\frac{\mu_l(q)}{a_l}\sin(a_lt)
&=&a\left(\cos(a_it)-\frac{\mu_i(q)}{a_i}\sin(a_it)\right)\\
&+&b\left(\cos(a_jt)-\frac{\mu_j(q)}{a_j}\sin(a_jt)\right).
\end{eqnarray*}
It is standard to show that the above functional equality, since $a\neq 0 \neq b$, implies that
$a_i=a_j=a_l$. This implies that
\begin{equation}\label{eq:lemma7b}
\Vert (\eta_i)_q\Vert=\Vert(\eta_j)_q\Vert=\Vert(\eta_l)_q\Vert.
\end{equation}
Applying $\pi_2$ to (\ref{eq:lemma7a}), we obtain
\begin{equation}\label{eq:lemma7c}
(\eta_l)_q = a(\eta_i)_q + b(\eta_j)_q.
\end{equation}
Thus, $(\eta_i)_q$, $(\eta_j)_q$, and $(\eta_l)_q$ lie on a line. But three
different points on a line cannot have the same length, contradicting
(\ref{eq:lemma7b}). This contradiction proves the lemma.
\end{proof}

\begin{lemma}\label{lem:I_0}
We keep the notation and assumptions of this section.
Let $N'$ be a local irreducible extrinsic factor of $N$, and let
$J\subset I$ be defined by the property that  $j\in J$ if and only if
$\tilde\eta_{j|N'}$ is a curvature normal of $N$ corresponding to $N'$.
Then $J\cap I_1$ is non-empty.
\end{lemma}
\begin{proof}
Assume, to the contrary, that this intersection is empty and thus $J\subset I_0$.
Hence, $f_{q,j}(t)\equiv 1$ for all $j\in J$ and $q\in N'$.
Thus, $\gamma_q(t)-q$ is a curve which is perpendicular to any curvature normal associated to $N'$.
Let us write, extrinsically,
$$N=N'\times N''\qquad \quad \text{(locally)}.$$
Then, for every (small) $t$, the parallel normal field $\xi^t_{|N'}$ of $N'$ satisfies
$$A'_{\xi^t}=0,$$
where $A'$ is the shape operator of $N'$ (see (\ref{eq:campo})). Then $\xi^t$ is constant along $N'$.
Observe that
$$
M=\bigcup_{t\in(-\varepsilon,\varepsilon)}N_{\xi^t}\qquad \quad  \ \text{(locally)}.
$$
Then, since $\xi^t$ is constant along $N'$,
$$
M=N'\times\bigcup_{t\in(-\varepsilon,\varepsilon)}(N'')_{\xi^t}\quad\text{(locally)},
$$
contradicting the local irreducibility of $M$.
\end{proof}
\medskip

\medskip

\begin{corollary}\label{cor:exFactor}
Let $N'$ be an extrinsic local factor of $N$ which has constant principal
curvatures. Then $\rank(N')=1$, where the rank is with respect to the
affine subspace spanned by $N'$.
\end{corollary}
\begin{proof}
The proof follows from Lemma~\ref{lem:I_0}, Corollary~\ref{cor:cpc},
Remark~\ref{rem:A4}, and Lemma~\ref{lem:line}.
\end{proof}

\medskip

A parallel normal field $\xi$ on a Euclidean submanifold $M$ is called a \textit{parallel normal  isoparametric  section} if the shape operator $A_\xi$ has constant eigenvalues. 
\medskip

The following local result follows immediately from
Theorems~4.5.10 and~4.5.2 and Corollary~4.5.3 of \cite{BCO}.
A global version can be found in Section~4.5.2 of the same reference,
but we will not make use of it.

\begin{theorem}[{isoparametric rank theorem} \cite{BCO}] \label{thm:iso-rank}
Let $M$ be a locally irreducible Euclidean submanifold. Assume that $M$
admits a parallel normal isoparametric  section $\xi\neq 0$.
Then
\begin{enumerate}
    \item $M$ is contained in a sphere.
    \item If $\xi$ is non-umbilical, then $M$ is a submanifold with constant principal curvatures of rank at least $2$ (and so $M$ is either an isoparametric
    hypersurface of a sphere or an (open subset of an) orbit of an $s$-representation).
\end{enumerate}
\end{theorem}
\medskip

Lemma~\ref{lem:acople}, together with Lemmas~\ref{lem:2.1}(1)
and~\ref{lem:4.1}(1), implies the following corollary.

\begin{corollary}\label{cor:uw8}
The restriction of $\eta_k$ to $N$ is a parallel normal isoparametric 
section. Namely,
\begin{equation*}
\bar A_{\eta_k|\bar E_j}=a_j^2\,\mathrm{Id}_{\bar E_j}
\text{ if } j\in I_1;\qquad
\bar A_{\eta_k|\bar E_j}=0
\text{ if } j\in I_0,
\end{equation*}
where $\bar A$ is the shape operator of $N$ and $\bar E_j=E_{j|N}$.
\end{corollary}

\smallskip

\begin{lemma}\label{lem:Main}
Let $N'$ be an irreducible extrinsic local factor of $N$. Let
$\bar\eta=(\eta_k)_{|N'}$, and let $\bar A$ denote the shape operator of
$N'$. Then
\begin{enumerate}
    \item $\bar A_{\bar\eta}$ is a non-zero scalar multiple of the
    identity of $TN'$ (perhaps after replacing $N$ by a nearby integral manifold).
    
    \item $N'$ is contained in a sphere.
    
    \item $\nu_0N'$ has only one associated curvature normal, or
    equivalently, $TN'$ is the only $\nu_0N'$-eigendistribution. 

    \item The curvature normal, associated with $N'$, has constant length, when projected to the parallel codimension one subbundle 
    $(X^\perp)_{|N'} = (\nu _0M)_{|N'}$ of $\nu _0N'$ (in the full ambient space).

    \item If $\dim N'\geq2$, then $\dim\nu_0N'=1$ (here $\nu_0N'$
    is regarded in the affine space spanned by $N'$). 
    
    \item $N$ is contained in a sphere.
\end{enumerate}
\end{lemma}
\begin{proof} By Corollary~\ref{cor:uw8}, $\bar\eta$ is a parallel isoparametric
normal section of $N'$, and hence the eigenvalues of $\bar A_{\bar\eta}$
are constant.
If $\bar A_{\bar\eta}$ has at least two distinct eigenvalues, then, by
Theorem~\ref{thm:iso-rank}, $N'$ is a submanifold with constant principal
curvatures and rank at least $2$. This contradicts
Corollary~\ref{cor:exFactor}. Therefore, $\bar A_{\bar\eta}$ has only one
eigenvalue, and hence
\(
\bar A_{\bar\eta}=c\,\mathrm{Id}.
\)
Let us show that $c\neq0$ (perhaps by considering a nearby integral
manifold of $X^\perp$). Let us decompose extrinsically
$N=N'\times N''$. If $N$ is not full, as in the local example in
\cite{DO}, then $N'' $, which may be a point, is not full in its ambient
space. Locally, $M$ is a union of parallel manifolds. Namely,
\[
M=
\bigcup_{t\in(-\epsilon,\epsilon)}N_{\xi^t}
\qquad\text{(locally)}.
\]
If $c=0$, then $\eta_k$ is perpendicular to any curvature normal of $N$
corresponding to $N'$. Then
\[
M=
N'\times
\bigcup_{t\in(-\epsilon,\epsilon)}(N'')_{\xi^t}
\qquad\text{(locally)},
\]
and thus $M$ is reducible, a contradiction. Thus $c\neq0$, which proves (1). Note that 
\(
z=q+c^{-1}\bar\eta_q
\)
does not depend on $q\in N'$. Hence $N'$ is contained in the sphere
with centre $z$ and radius $\|c^{-1}\bar\eta\|$. This proves (2).
If $\tilde\eta_j$ is any  curvature normal of $N'$, then 
$$c = \langle \bar\eta , \tilde\eta _j\rangle = \langle \eta _k , \tilde\eta _j\rangle 
 = \langle \eta _k , \eta _j\rangle  = \langle \eta _j, \eta _j\rangle ,$$
 where the last equality is due to Lemma \ref{lem:acople}. Then all the curvature normals of $N'$, with respect to the parallel
and flat subbundle $\nu_0M_{|N'}$ of $\nu_0N'$, have the same length.
By Remark~\ref{rem:equalL}, since $N'
$ is locally irreducible, either
there is only one eigendistribution, or $N'$ is a submanifold with
constant principal curvatures and rank at least $2$. The latter
possibility is excluded by Corollary~\ref{cor:exFactor}. This proves (3). Part (4) follows from Corollary~\ref{cor:11}.
It follows from Remark~\ref{rem:equalL2} that (5) holds.

Since every extrinsic factor of $N$ is contained in a sphere, $N$ is
contained in a sphere, and hence (6) follows.
\end{proof}
\begin{corollary}\label{cor:Main}
Assume that $M$ has flat normal bundle and let $N_1$ be an irreducible
extrinsic local factor of $N$ of dimension at least $2$. Then $N_1$ is an
extrinsic sphere in the ambient Euclidean space.
\end{corollary}

\begin{remark}\label{rem:admissible}
Let us keep the assumptions and notation of Lemma \ref{lem:Main}. Let $N'$ be an extrinsic factor of $N$, which must be contained in a sphere, and let $\mathbb A$ be its affine span.
Let us write
$$\nu _0N'= \bar{\nu}_0N'\oplus \mathbb R^s,$$
where $\bar{\nu}_0N'$ is the flat part of the normal bundle of $N'\subset \mathbb A$ and $\mathbb R^s$ is to be regarded as the pull-back of the normal bundle of $\mathbb A\subset \mathbb R^{n+p}$. Namely, $\mathbb R^s$ corresponds to the normal fields of $N'$ that are constant in the ambient space. The unit field $X_{|N'}$ is a parallel section of $\nu _0N'$, which does not in general lie in $\bar{\nu}_0N'$, as the examples in \cite[Sec. 5]{DO} show. Then $X$ is a parallel section of the bundle $\mathbb RY\oplus \bar{\nu}_0N'$, where $Y$ is the component of $X$ in $\mathbb R^s$.
If $Y\neq 0$, then $\mathbb R\simeq \mathbb RY$ is a flat one-dimensional subbundle of $\nu _0N'$ and $X$ is a parallel section of
\begin{equation}\label{eq:145}
\bar{\nu}_0N'\oplus \mathbb R.
\end{equation}
If $Y=0$ we choose $\mathbb R$ to be a parallel one-dimensional subbundle of
$\mathbb R^s$ (perhaps after immersing $N'$ as a non-full submanifold).
In any case, $X_{|N'}$ is a parallel section of the bundle (\ref{eq:145}).
Let $\hat{\eta}$ be the (unique) curvature normal field associated to $N'$. Then, by Lemma \ref{lem:Main}(4), the component $\hat{\eta} -\langle \hat{\eta},X\rangle X$ of $\hat{\eta}$ perpendicular to $X$ has constant length.
This is the motivation for the definition below, which we only need for one-dimensional $N'$.
\end{remark}

\begin{definition}\label{def:6} Let $C$ be a full (local) one-dimensional submanifold of $\mathbb R^m$, which is contained in a sphere. Let us regard $C$ as a non-full submanifold of $\mathbb R^{m+1}$, and let $\nu_0C$ denote its flat normal bundle computed in $\mathbb R^{m+1}$. The submanifold $C$ is called \emph{admissible} if there is a parallel section $Y\neq 0$ of $\nu _0C$ such that the projection to $Y^\perp$ of the curvature normal of $C$ has constant length.
\end{definition}

If the curvature normal of $C$ has constant length, then $C$ is admissible. In fact, It suffices to choose $Y$ as any linear combination of $e_{m+1}$ and  the position vector $\vec p$ of $C$, with respect to the centre of the sphere that contains $C$. 

\section{Inner products of curvature normals}\label{sec:inner}

We begin this section with an auxiliary result that will be used later. 

\begin{lemma}\label{lem:9f7}
    If $I_0\neq \emptyset $,  then $|I_1|\geq 2$.
\end{lemma}
\begin{proof}
By Remark~\ref{rem:820} or Lemma~\ref{lem:I_0}, $I_1\neq\emptyset $.
Assume that $I_1=\{i\}$. By the Codazzi identity, the distributions
\[
\bar E_i^\perp=\ker\bar A_{\eta_k}
\quad\text{and}\quad
\bar E_i=\ker(\bar A_{\eta_k}-a_i^2\mathrm{Id})
\]
are orthogonally complementary autoparallel distributions such that
\(
\bar\alpha(\bar E_i,\bar E_i^\perp)=0,
\)
where $\bar\alpha$ is the second fundamental form of $N$. Hence, by
Moore's lemma, $N$ locally splits extrinsically as
$N=N_0\times N_1$, where $N_0$ is an integral manifold of
$\bar E_i^\perp$ and $N_1$ is an integral manifold of $\bar E_i$.
Locally, $M$ is the union of parallel manifolds of $N$. Namely,
\[
M=\bigcup_{t\in(-\epsilon,\epsilon)}N_{\xi^t}
\qquad\text{(locally)},
\]
where $\xi^t_q=\gamma_q(t)-q$ (see the beginning of
Section~\ref{sec:structure}). Observe that the curvature normals of $N$
corresponding to $N_0$ are
\(
\{\tilde\eta_{i|N}:i\in I_0\}.
\)
But such curvature normals are always perpendicular to
$\gamma_q(t)-q$, since $f_{q,i}(t)$ is identically $1$ for all $q\in N$.

It is now standard to show that $M$ extrinsically splits as
\[ M= 
N_0\times
\bigcup_{t\in(-\epsilon,\epsilon)}(N_1)_{\xi^t}
\qquad\text{(locally)}.
\]
This is a contradiction, since $M$ is irreducible.
\end{proof}

\medskip

Let us compute the inner product between $(\eta_i)_{\gamma_q(t)}$ and
$(\eta_j)_{\gamma_q(t)}$, where $i,j\in I_1$ and $i\neq j$.
From (\ref{eq:84}), taking into account (\ref{eq:002}) and the fact that
$\gamma_q(t)$ has unit speed, we obtain

\begin{equation}\label{eq:<,>}
  \langle (\eta_i)_{\gamma_q(t)}, (\eta_j)_{\gamma_q(t)}\rangle =  
    \frac {\langle (\tilde \eta_i)_q,  
    (\tilde \eta_j)_q\rangle  - f'_{q,i}(t)f'_{q,j}(t)}{f_{q,i}(t)f_{q,j}(t)}
\end{equation}

This function is constant if and only if
\begin{equation}\label{eq:<,>2}
    \frac {\langle (\tilde \eta_i)_q,
    (\tilde \eta_j)_q\rangle - f'_{q,i}(t)f'_{q,j}(t)}
    {f_{q,i}(t)f_{q,j}(t)}
    = \lambda := \langle (\eta_i)_q,(\eta_j)_q\rangle
\end{equation}
if and only if
\begin{equation}\label{eq:<,>3}
    \langle (\tilde \eta_i)_q,
    (\tilde \eta_j)_q\rangle - f'_{q,i}(t)f'_{q,j}(t)
    = \lambda f_{q,i}(t)f_{q,j}(t)
\end{equation}
if and only if the derivatives of the left- and right-hand sides coincide,
since both sides coincide at $t=0$. Namely, 
\begin{equation}\label{eq:<,>4}
     - f''_{q,i}(t)f'_{q,j}(t) - f'_{q,i}(t)f''_{q,j}(t)
    = \lambda f '_{q,i}(t)f_{q,j}(t) + \lambda f_{q,i}(t)f '_{q,j}(t)
\end{equation}
By (\ref {eq:expsol}), $f''_{q,i}(t)= -a_i^2f_{q,i}(t)$ and that $f''_{q,j}(t)= -a_j^2f_{q,j}(t)$. Then 
\begin{equation}\label{eq:<,>5}
      a_i^2f_{q,i}(t)f'_{q,j}(t) + a_j^2 f'_{q,i}(t)f_{q,j}(t)
    = \lambda f '_{q,i}(t)f_{q,j}(t) + \lambda f_{q,i}(t)f '_{q,j}(t) 
\end{equation}
or, equivalently, 

\begin{equation}\label{eq:<,>6}
      (a_i^2 -\lambda) f_{q,i}(t)f'_{q,j}(t) 
    = -( a_j^2 - \lambda) f_{q,j}(t)f '_{q,i}(t) 
\end{equation}

We may assume that $i<j$, and thus, by our convention,
$\Vert \eta_i\Vert=a_i\geq a_j=\Vert \eta_j\Vert$.
From the definition of $\lambda$ and the Cauchy-Schwarz inequality,
$\lambda\leq a_i a_j$, with equality if and only if
$(\eta_j)_q$ is a nonnegative scalar multiple of $(\eta_i)_q$. Then
$a_i^2=\lambda$ implies that $a_i=a_j$ and that equality holds in the
Cauchy-Schwarz inequality. This contradicts Lemma \ref{lem:acople}(3). Hence, $a_i^2>\lambda$, and therefore
\begin{equation}\label{eq:45A}
\alpha_i:=a_i^2-\lambda\neq0.
\end{equation}
This also implies, from (\ref{eq:<,>6}), that
\begin{equation}\label{eq:45B}
\alpha_j:=a_j^2-\lambda\neq0.
\end{equation}
Observe that $\alpha_i=\alpha_j$ if and only if $a_i=a_j$.

\smallskip

The calculations in Section~\ref{auxCal} show that the equality (\ref{eq:<,>6}) never holds,
thus proving the following lemma.

\begin{lemma}\label{lem:non<,>}
If $i,j\in I_1$, $i\neq j$, then $\langle\eta_i,\eta_j\rangle$ is not constant along
$\gamma_q(t)$ and therefore is not constant on $M$.
\end{lemma}


\begin{corollary}\label{cor:algCons}
Let $M$ be an irreducible and full submanifold of Euclidean space with
$\rank(M)\geq 2$. Assume that every curvature normal of $M$ with respect
to $\nu_0M$ has constant length and that $M$ is not a submanifold with  constant
principal curvatures.
Then the inner product of any pair of curvature normals is constant if
and only if $I=I_1$ and $|I_1|=1$.
\end{corollary}
\begin{proof}
Assume that $I=I_1$ and $|I_1|=1$. Then, in our notation, there are only
two curvature normals: $\eta_k$, the longest non-parallel curvature normal,
and $\eta_i$, where $k\in\{1,2\}$ and
$i\in I_1=\{1,2\}\setminus\{k\}$ (see Remark~\ref{rem:123}).
Then $\langle\eta_i,\eta_k\rangle$ is constantly $a_i^2$ by
Lemma~\ref{lem:acople}(2), proving the if part.

Let us prove the converse. If $|I_1|\geq2$, let $i,j\in I_1$ with
$i\neq j$. In particular, this always holds if $I_0\neq\emptyset $, by
Lemma~\ref{lem:9f7}. Then $\langle\eta_i,\eta_j\rangle$ is not constant
by Lemma~\ref{lem:non<,>}. Therefore, if the inner product of any two curvature normals is constant,
then $|I_1|=1$. This implies that $I_0=\emptyset $ and hence
$I=I_1$.
This completes the proof.
\end{proof}

\begin{corollary}\label{cor:algCons2}
Let $M$ be an irreducible and full submanifold of Euclidean space with
flat normal bundle. Assume that the second fundamental form is algebraically constant and that $M$ is not an isoparametric submanifold. 
Then $M$ has positive constant sectional curvature. 
\end{corollary}
\begin{proof} If the second fundamental form is algebraically constant, then
the inner product of any two curvature normals is constant. 
     Then, by Corollary~\ref{cor:algCons}, at any point of $M$ there are exactly two different eigenspaces, let us say $E_1=\operatorname{span}\{e_1\}$, with $\|e_1\|=1$, and 
$E_2$  a subspace of dimension $n-1$, and $2\in I_1 =I$. Then associated curvature normals  $\eta _1,\eta _2$, are related by the condition $\langle \eta _1,\eta _2\rangle = \langle \eta _2, \eta _2\rangle$ given by Lemma~\ref{lem:acople}.
We have also that 
\[
\alpha(E_1,E_2)=0,\qquad
\alpha(e_1,e_1)=\eta_1, \qquad
\alpha(X,X)=\Vert X\Vert^2 \eta_2 \, ,
\]
for every $X\in E_2$. 

Let us show that the sectional curvatures are constant.
Let $V=E_1\oplus E_2$, and let $u,v\in V$ be arbitrary
orthonormal vectors. Write
\[
u=ae_1+X,\qquad v=be_1+Y,
\]
where $X,Y\in E_2$. Since $u$ and $v$ are orthonormal, we have
\[
a^2+\|X\|^2=1\, ,\qquad
b^2+\|Y\|^2=1 \, ,\qquad
ab+\langle X,Y\rangle=0.
\]

Since $\alpha(E_1,E_2)=0$, it follows that
\[
\alpha(u,v)
=ab\,\eta_1+\alpha(X,Y)
=ab\,\eta_1+\langle X,Y\rangle\eta_2\, .
\]
Using $\langle X,Y\rangle=-ab$, we obtain
\[
\alpha(u,v)=ab(\eta_1-\eta_2)\, .
\]
Moreover,
\[
\alpha(u,u)
=a^2\eta_1+\|X\|^2\eta_2
=a^2\eta_1+(1-a^2)\eta_2,
\]
and similarly
\[
\alpha(v,v)
=b^2\eta_1+(1-b^2)\eta_2.
\]

Therefore, by the Gauss equation,
\[
\begin{aligned}
K(u,v)
&=
\langle\alpha(u,u),\alpha(v,v)\rangle
-\|\alpha(u,v)\|^2\\
&=
\left\langle
a^2\eta_1+(1-a^2)\eta_2,\,
b^2\eta_1+(1-b^2)\eta_2
\right\rangle\\
&\qquad
-a^2b^2\|\eta_1-\eta_2\|^2.
\end{aligned}
\]
Expanding the right-hand side gives
\[
\begin{aligned}
K(u,v)
={}&a^2b^2\|\eta_1\|^2
+a^2(1-b^2)\langle\eta_1,\eta_2\rangle\\
&+(1-a^2)b^2\langle\eta_1,\eta_2\rangle
+(1-a^2)(1-b^2)\|\eta_2\|^2\\
&-a^2b^2
\left(
\|\eta_1\|^2
-2\langle\eta_1,\eta_2\rangle
+\|\eta_2\|^2
\right).
\end{aligned}
\]
Taking into account that $\langle\eta _1 ,\eta _2\rangle = \langle\eta _2 ,\eta _2\rangle$, we obtain
\[
K(u,v)
= \langle\eta _2, \eta _2\rangle
\]

Thus every $2$-plane in $E_1\oplus E_2$ has the same sectional
curvature.
\end{proof}


\subsection{Auxiliary calculations} \label{auxCal}
Just for the sake of simplifying the notation we set $i=1$ and  $j=2$.
\[
f_l(t)=\cos(a_l t)-\frac{b_l}{a_l}\sin(a_l t),
\qquad l=1,2,
\]
where $a_1\neq0$. Then
\[
f_l'(t)=-a_l\sin(a_l t)-b_l\cos(a_l t).
\]

Define
\[
g(t)=f_1(t)f_2'(t),
\ \ 
P(t)=\frac12\bigl(g(t)+g(-t)\bigr) \text{ (even part)},
\ \ 
I(t)=\frac12\bigl(g(t)-g(-t)\bigr) \text{ (odd part)}.
\]
A direct computation gives
\[
P(t)
=
-b_2\cos(a_1t)\cos(a_2t)
+\frac{b_1a_2}{a_1}\sin(a_1t)\sin(a_2t),
\]
and
\[
I(t)
=
-a_2\cos(a_1t)\sin(a_2t)
+\frac{b_1b_2}{a_1}\sin(a_1t)\cos(a_2t).
\]

After swapping  the indices $1$ and $2$, define
\[
\bar g(t)=f_2(t)f_1'(t),
\qquad
\bar P(t)=\frac12\bigl(\bar g(t)+\bar g(-t)\bigr),
\qquad
\bar I(t)=\frac12\bigl(\bar g(t)-\bar g(-t)\bigr).
\]
Then
\[
\bar P(t)
=
-b_1\cos(a_1t)\cos(a_2t)
+\frac{b_2a_1}{a_2}\sin(a_1t)\sin(a_2t),
\]
and
\[
\bar I(t)
=
-a_1\sin(a_1t)\cos(a_2t)
+\frac{b_1b_2}{a_2}\cos(a_1t)\sin(a_2t).
\]

\smallskip

If the equality (\ref{eq:<,>6}) holds, then the even (resp. odd) part of
the left-hand side coincides with the even (resp. odd) part of the
right-hand side. 

\smallskip

Note, from (\ref{eq:45A}) and (\ref{eq:45B}), that 
$\alpha_1\neq 0\neq\alpha_2$.

\smallskip

\begin{itemize}
 \item[(1)] Case $a_1\neq a_2$ (which implies $\alpha _1\neq \alpha _2$).
\end{itemize}
The condition
\[
\alpha_1P(t)=-\alpha_2\bar P(t)
\]
gives, using the independence of
$\cos(a_1t)\cos(a_2t)$ and
$\sin(a_1t)\sin(a_2t)$,
\[
{
\begin{cases}
\alpha_1b_2+\alpha_2b_1=0,\\[1mm]
\alpha_1b_1a_2^2+\alpha_2b_2a_1^2=0.
\end{cases}}
\]

For the odd parts,
\[
\alpha_1I(t)=-\alpha_2\bar I(t).
\]
Since $a_1\neq a_2$,
\[
\cos(a_1t)\sin(a_2t)
\quad\text{and}\quad
\sin(a_1t)\cos(a_2t)
\]
are linearly independent. 
Hence
\[
{
\begin{cases}
\alpha_1a_2^2=\alpha_2b_1b_2,\\[1mm]
\alpha_1b_1b_2=\alpha_2a_1^2.
\end{cases}}
\]

Thus the four conditions are
\[
{
\begin{cases}
\alpha_1b_2+\alpha_2b_1=0,\\
\alpha_1b_1a_2^2+\alpha_2b_2a_1^2=0,\\
\alpha_1a_2^2=\alpha_2b_1b_2,\\
\alpha_1b_1b_2=\alpha_2a_1^2.
\end{cases}}
\]
There is no nontrivial solution when $a_1,a_2>0$ (as in our situation).
Indeed, the third and fourth equations show that 

\[
-\frac{b_1}{b_2}
=
\frac{\alpha_1}{\alpha_2}
=
\frac{b_1b_2}{a_2^2}.
\]
Using the third and the forth equality we obtain that  $b_1,b_2\neq0$. This implies
\[
-b_2^2=a_2^2,
\]
which is impossible for real $b_2$ and $a_2>0$.
Consequently, the four conditions are incompatible.

\begin{itemize}
    
 \item[(2)] Case $a_1=a_2=:a>0$ (which implies that $\alpha_1=\alpha_2 =:\alpha\neq 0$).
 
\end{itemize}
Now
\[
P(t)
=
-b_2\cos^2(at)+b_1\sin^2(at),
\]
and
\[
\bar P(t)
=
-b_1\cos^2(at)+b_2\sin^2(at)
\]
(observe that $\cos ^2(at)$ and $\sin^2(at)$ are linearly independent).
Therefore, equality (\ref{eq:<,>6}) implies that 
\[
\alpha P(t)=-\alpha \bar P(t)
\]
gives only the condition $\alpha b_1 + \alpha b_2= 0$, or equivalently 
\begin{equation}\label{eq:cont} b_1+b_2=0.
\end{equation}

For the odd parts, the two functional terms coincide. Namely, 
\[
I(t)=\bar I(t)
=
\left(\frac{b_1b_2}{a}-a\right)\sin(at)\cos(at).
\]
Hence the equality
\[
\alpha I(t)=-\alpha \bar I(t)
\]
gives only one condition,
\(
{
2\alpha(b_1b_2-a^2)=0
}
\). Equivalently, 
\begin{equation*}
    b_1b_2=a^2
\end{equation*}
Then, from (\ref{eq:cont}), $-b_1^2=a^2$, which is a contradiction,
since $a\neq0$. \qed

\section{Proofs of the main results}\label{sec:proof}

\subsection{Proof of Theorem~\ref{thm:main1}}\ 
\noindent Parts (1) and (2) are just Lemma~\ref{lem:4.1}. Part (3) is Lemma~\ref{lem:4.2}. Part (4) is a consequence of (3) (see the proof of Lemma~\ref{lem:Main}). The proof of (5) is given in the discussion following
(\ref{eq:u8e}) in Section~\ref{rec}. Part (6) follows from Corollary~\ref{cor:11}. Part (7) follows from (6) and part (8) follows from Corollary~\ref{cor:uw8}. Part (9) is just Lemma~\ref{lem:Main}(\textit{6}). Part (10) follows from (6) and Lemma~\ref{lem:Main}(\textit{3}). 

\subsection {Proof of Theorem~\ref{thm:main1B}}\ 
It follows  from Theorem~\ref{thm:main1}(7)(10) and Lemma~\ref{lem:Main}(\textit{2}),(\textit{4}), (\textit{5}).

\subsection{Proof of Theorem~\ref{thm:main2} and Corollary~\ref{cor:main3}}
The theorem follows from Corollary~\ref{cor:algCons} and part (1) of Theorem~\ref{thm:main1}. The corollary follows from this theorem and Corollary~\ref{cor:algCons2}.

\section{General examples}\label{sec:examples}

The purpose of this section is to show that submanifolds $N$ as in
Theorem~\ref{thm:main1B}(A), together with one-dimensional
submanifolds contained in a sphere whose curvature normal has constant
length, do occur as generating hypersurfaces: any product $N$ of such
factors is the generating hypersurface of a locally irreducible
submanifold $M$ with $\nu_0M$-curvature normals of constant length.

This is not a full converse of Theorem~\ref{thm:main1B}, since that
theorem only requires the one-dimensional factors to be admissible.
We do not know how to carry out the construction below when a
one-dimensional factor is admissible but its curvature normal does not
have constant length.

\smallskip

Let $N_i\subset\mathbb R^{n_i}$, $i=1,\dots,s$, be full submanifolds
that are either as in Theorem~\ref{thm:main1B}(A), or are
one-dimensional with its curvature normal of constant length and
contained in a sphere. Let $N=N_1\times \dots \times N_s$. The position vector field, with respect to the centre
of the corresponding sphere, is denoted by $\vec p_i$, and its curvature
normal is denoted by $\tilde\eta_i$. Let
\[
\xi_i=-\frac{1}{r_i}\vec p_i,\qquad r_i=|\vec p_i|.
\]
Note that if $N_i$ is as in (A), then
\[
\tilde\eta_i=\frac{1}{r_i}\xi_i.
\]
In either case, $0\in\mathbb R^{n_i}$ is an isolated focal point of
$N_i$, and hence
\[
\langle\tilde\eta_i,\vec p_i\rangle=-1.
\]
Observe that the parallel normal fields $\xi_1,\dots,\xi_s$ of $N$ are
orthonormal and that 
\begin{equation}\label{eq:49}
\begin{cases}
   \langle \tilde\eta _i , \xi _i\rangle = \frac{1}{r_i},  \\
\langle \tilde\eta _j , \xi _i\rangle=0, \text{ if } j\neq i. 
\end{cases}
\end{equation}

Let $\mu _i , a_i \in \mathbb R$, $\mu _i \neq 0$, $a_i>0$ be such that 

\smallskip

\begin{enumerate}  
\item $\mu_i^2 + a_i^2 = 
\vert \tilde\eta_i\vert ^2$
\item $\sum _{i=1}^{s}r_i^2\mu_i^2 <1$
\end{enumerate}
Let, for $q\in N$,  
\begin{equation}\label{eq:50}\beta ^i_q(t) : = r_i \left (1-\cos(a_it) + \frac{\mu_i}{a_i}\sin(a_it)\right)(\xi _i)_q, \end{equation}
\begin{equation}\label{eq:51}
\gamma _{q,1} (t): = \sum_{i=1}^s\beta ^i_q(t)
\end{equation}
Then 
\begin{equation}\label{eq:52}
|\gamma _{q,1}'(0)|^2 = \sum _{i=1}^{s}r_i^2\mu_i^2 <1
\end{equation}

By Remark~\ref{rem:gamma_2} below, for any fixed $q\in N$, there exists a
curve $\gamma_2(t)$ in $\mathbb R^2$ such that
\[
\gamma_q(t):=\big(\gamma_{q,1}(t),\gamma_2(t)\big)
\]
has unit speed and constant curvature length $c$, where $c$ can be
chosen arbitrarily large. We may assume that $\gamma_2(0)=0$ and hence
$\gamma_q(0)=0$. By its construction, $\gamma_2(t)$ does not depend on
$q$.

\smallskip

Observe that, for any fixed $t$,
\[
q\overset{\xi^t}{\longmapsto}\gamma_q(t)
\]
is a parallel normal field of $N\subset\mathbb R^n$, where
$n=n_1+\dots+n_s$.

\smallskip

Our curve $\gamma _q(t)$ has a slight difference with that in  Theorem~\ref{thm:main1} since it is a curve in the normal space of $N$ at $q$ and $\gamma _q(0) = 0 \in \nu_qN$. Thus, we have not to subtract the initial point in order to define the parallel normal field $\xi ^t$.

\medskip

From (\ref{eq:49}), (\ref{eq:50}), and (\ref{eq:51}) we obtain that 

\begin{equation}\label{eq:53}
  \langle \tilde\eta _i , \xi^t\rangle =  1-\cos(a_it) + \frac{\mu_i}{a_i}\sin(a_it)
\end{equation}
 and therefore
 \begin{equation}\label{eq:54}
  \langle (\tilde\eta _i)_q , \gamma _q'(t)\rangle =  {\mu_i}\cos(a_it) + a_i\sin(a_it)
\end{equation}
which does not depend on $q\in N$. Observe that $\mu_i$ is the component of $(\tilde\eta_i)_q$ in the
direction of $\gamma_q'(0)$. Hence, from our assumptions, $a_i$ is the
length of the component of $(\tilde\eta_i)_q$ which is perpendicular to
$\gamma_q'(0)$. 

Let $$f_{q,i}(t):= 1-  \langle (\tilde\eta _i)_q , \gamma _q(t)\rangle = 
\cos(a_it) - \frac{\mu_i}{a_i}\sin(a_it)
.$$
Then the curvature normals of 
$N_{\xi ^t}$ at $ q + \gamma _q(t)$ are 
\begin{equation}\label{eq:55}
 (\tilde \eta _i^t)_{q + \gamma _q(t)} = \tfrac{1}{1- \langle  (\tilde \eta _i)_q, \xi ^t_q \rangle} \, (\tilde \eta _i)_q 
 = \tfrac{1}{1- \langle  (\tilde \eta _i)_q, \gamma _q(t) \rangle} \, (\tilde \eta _i)_q = \frac{1}{f_{q,i}(t)}(\tilde \eta _i)_q 
\end{equation}
$i=1, \dots , s$.

Define $M$ as the union of parallel manifolds of $N$. Namely, 
\[
M=\bigcup_{t\in(-\epsilon,\epsilon)}N_{\xi^t}
\qquad\text{(locally)},
\]
The curvature normals of $M$ at $q+\gamma_q(t)$ are the projections
$\eta_i$ onto the normal space of $M$ of the curvature normals
$\tilde\eta_i^t$ of $N_{\xi^t}$. We have to add one more curvature
normal $\eta$, which at the point $q+\gamma_q(t)$ is given by
$\gamma_q''(t)$; this has constant length by the construction of
$\gamma_q(t)$.
The arguments are the same as those in the proof of
Lemma~\ref{lem:n_k}. The fact that $\eta_i$ has constant length follows
from the same calculation as the one given in (\ref{eq:a_i^2}) and the surrounding 
formulas. Thus, $M$ has $\nu_0M$-curvature normals of constant length, with
$g=s+1$ curvature normals in total: $\eta$ corresponds to $\eta_k$, and
$\eta_1,\dots,\eta_s$ correspond, in some order, to the remaining
curvature normals $\eta_i$, $i\in\{1,\dots,g\}\setminus\{k\}$, in the
notation used throughout the previous sections.

Observe that $f_{q,i}(t)$ is not constant, for all $i=1,\dots,s$. In the
notation of Section~\ref{sec:long}, we have $I_1=\{1,\dots,s\}$. Then,
$$
 \langle\eta_i,\eta\rangle=-f_i''/f_i=a_i^2\neq0.
$$

This implies that $M$ is not locally a product. Indeed, otherwise, the
curvature normals of $M$ would decompose into two subsets that are
mutually orthogonal.

This finishes the construction of the examples showing that an
arbitrary $N$ of this type is a generating hypersurface of a submanifold
with curvature normals of constant length. We summarize the
construction of this section in the following proposition.

\begin{proposition}\label{prop:examples}
Let $N=N_1\times\dots\times N_s$, where each $N_i\subset\mathbb R^{n_i}$
is a full submanifold that is either as in Theorem~\ref{thm:main1B}(A)
or one-dimensional, contained in a sphere, and with curvature normal of
constant length. Then $N$ is (locally) the generating hypersurface of a
locally irreducible submanifold $M$, with $g=s+1$ curvature normals
with respect to $\nu_0M$, all of constant length.
\end{proposition}

\smallskip

\begin{remark}\label{rem:gamma_2}
We recall the following elementary fact. Given a $C^2$ curve $\gamma_1:I\to\mathbb R^n$ with $0<|\gamma_1'(0)|<1$, after restricting $I$ if necessary, one can find a $C^2$ curve $\gamma_2:I\to\mathbb R^2$ such that the curve
$$
\gamma=(\gamma_1,\gamma_2):I\to\mathbb R^{n+2}
$$
has unit speed and $|\gamma''|$ is constant. Moreover, this constant can be chosen arbitrarily large. Indeed, set

$$
s(t)=\sqrt{1-|\gamma_1'(t)|^2}.
$$

and let $\gamma '_2(t) = s(t)(\cos (\theta (t)), \sin(\theta (t)))$, where 
$\theta(t)$ has to be determined from the condition that $\gamma (t) = (\gamma _1(t),\gamma _2(t))$ has curvature of constant length. This gives 
$$\theta '(t) = \frac{1}{s(t)}\sqrt{c^2 - \gamma _1'' (t)^2 - s'(t)^2 }.$$
Integrating this expression we obtain 
$\theta (t)$ (we can choose the initial condition to be $\theta (0) = 0$. 
Then $$\gamma _2(t) =\int _0^t s(t)\big( \cos(\theta (t)), \sin(\theta(t))\big)$$ satisfies 
that the curve $\gamma (t) = (\gamma _1(t), \gamma _2(t))$ has unit speed and curvature of constant length $c$.
\end{remark}

\appendix\section{Submanifolds with constant principal curvatures}\label{App:A}

In this appendix, we collect some basic facts about submanifolds with
constant principal curvatures. Some of these results do not seem to be
available in the literature, and we include full proofs.

\medskip

Let us recall some standard definitions.  The \textit{rank} of a submanifold $M$ of Euclidean space, denoted by $\rank (M)$, is the dimension of $\nu_0M$, where $\nu _0M$ is the flat part of the normal bundle (see \cite{DO, BCO}).  The submanifold $M$ is said to have \textit{constant principal curvatures} if, along every curve in $M$, the shape operator with respect to any parallel normal field has constant eigenvalues.
Submanifolds with constant principal curvatures are either isoparametric if the normal bundle is flat, or focal parallel manifolds of isoparametric submanifolds 
(see \cite{HOT, BCO}). From Thorbergsson's theorem \cite{Th} one obtains that any complete irreducible and full submanifold with constant principal curvatures and rank at least $2$ is either an inhomogeneous isoparametric hypersurface of the sphere or an orbit of an irreducible $s$-representation (any such local submanifold extends to a compact embedded one, up to an extrinsic Euclidean factor). 

\smallskip

Let $K$ act on $\mathbb R ^N$ as an irreducible $s$-representation. Then any non-trivial $K$-orbit is an irreducible and full submanifold.
For the sake of clarity, we review some general known facts. 
We next describe each $K$-orbit and its rank with respect to an arbitrarily fixed principal orbit $K\cdot q$.
Let $\eta _1 , \dots , \eta _d$ be the curvature normals at $q$ associated to the irreducible isoparametric submanifold $K\cdot q$. Let $\ell _i$ be the hyperplane of $\nu _qM$ defined by the equation $\langle \eta _i , \cdot \rangle = 1$, $i=1, \dots , d$. 
The finite group generated by the reflections across all of these hyperplanes is the so-called \textit{Weyl} group associated with the isoparametric submanifold $K\cdot q$ \cite{PT} (see also \cite{BCO}). A \textit {Weyl chamber} $\bar U$ is the closure of a connected component 
of $\nu _q(K \cdot q)\setminus \bigcup _{i=1}^d\ell _i$. The interior of any two distinct Weyl chambers are disjoint. The Weyl group acts simply transitively on the set of Weyl chambers.  Let $\bar U$ be a fixed Weyl chamber. Then,  any $K$-orbit $M$ determines  a unique $p\in \bar U$ such that $M=K\cdot p$. Let $p\in \sigma$, where $\sigma$ is the $k$-simplex of $\bar U$ that contains $p$. Let 
$$I(\sigma)  = \{i, 1\leq i\leq d: \sigma \subset \ell _i\}\, ,  \quad \quad
V(\sigma) = \bigcap _{i\in I}\ell _i \, . $$ 
We have that $(\nu _0M)_p = V(\sigma)$ and hence $\rank (M)= \dim V(\sigma) = k$. If $\xi\in V(\sigma)$ is small, then $p+ \xi \in \sigma$.

The normal holonomy representation of $M$ at $p$ coincides with the slice representation  
of the isotropy group $K_p$ on $\nu _pM$ (see \cite{HO, BCO}). Observe that any 
$k\in K_p$ maps the simplex $\sigma$ into the simplex $k\sigma$ . Since both simplex contains $p$, we conclude that $k\sigma = \sigma$. If $z\in \sigma$, then $K\cdot (kz) = K\cdot z$. Since any $K$-orbit has a unique representative in $\bar U$, we conclude that $kz=z$ and thus $K_p$ fixes $\sigma$ pointwise. Since for a small $\xi \in (\nu _0M)_p$, $p + \xi \in \sigma$,  we obtain that $K_p$ fixes $(\nu _0M)_p$ pointwise. This implies that $\nu _0M$ is globally flat. 

Since, for any orbit $K\cdot p$ of an $s$-representation, the parallel transport in the normal bundle is given by $K$, every curvature normal $\eta_1,\dots,\eta_g$ is $K$-invariant (see the main theorem and Definition 1.2 of \cite{OS}; see also Theorem 4.4.3 of \cite{BCO}).

Let $M^n$ be a compact (embedded) full and irreducible submanifold  of $\mathbb R ^N$ with constant principal curvatures and rank at least $2$, i.e. 
$\rank (M) =\dim \nu _0M\geq 2$. We may assume that $M$ is complete.
The bundle $\nu _0M$ is globally flat. In fact, if   $M$ is an isoparametric hypersurface of the sphere, then $\nu _0M$ coincides with the globally flat normal bundle $\nu M$ in the Euclidean space. If $M$ is not an isoparametric hypersurface of the sphere, then 
 $M=K\cdot p $, where $K$ acts on $\mathbb R ^N$ as  an irreducible  $s$-representation. 
Let $\xi\in (\nu _0M)_p^\perp $ be a (short) principal vector for the 
normal holonomy group of $M$ at $p$, 
or equivalently, a principal vector for the slice representation $K_p$ on $(\nu _0M)_p^\perp$.  Then the holonomy tube 
$K\cdot (p+\xi)$ has a flat normal bundle and hence it is isoparametric. The assertion follows from the fact that isoparametric submanifolds have globally flat normal bundle. 

Each curvature normal $\eta_1,\ldots,\eta_g$ of $M$ with respect to $\nu_0M$ is a (globally defined) parallel section of $\nu_0M$. The curvature normals have an associated decomposition into eigendistributions,
$$
TM = E_1\oplus \cdots \oplus E_g
$$
where each eigendistribution is autoparallel and invariant under all shape operators of $M$ (equivalently, $\alpha(E_i,E_j) = 0$ if $i\neq j$). The curvature normals span $\nu_0M$, since any parallel normal field perpendicular to this span must be constant in the ambient space.  Observe that $g\geq 2$, since the rank of $M$ is at least $2$.

\medskip

Let us denote by $S_i(p)$ the integral manifold of $E_i$ that contains $p$. Observe that $S_i(p)$ is a submanifold with constant principal curvatures of the ambient space, invariant under all shape operators of $M$ restricted to $S_i(p)$. Moreover,  it  is contained in a proper affine subspace. Namely, 
\begin{equation}\label{eq:proper9}
S_i(p) \subset p + E_i(p) \oplus \nu _pM \subset \mathbb R^{n+p} .
\end{equation}

\begin{lemma}\label{lem:398}
Let $M$ be a complete full submanifold with constant principal curvatures of Euclidean space, with curvature normals $\eta_1,\dots,\eta_g$ and eigendistributions $E_1,\dots,E_g$ with respect to $\nu_0M$. Then, for every $p\in M$, $i=1, \dots , g$, 
\begin{enumerate}
    \item \  
$\displaystyle{(\eta_i)_p=\frac{1}{\dim E_i}\,H_i(p)}$.
where $H_i(p)$ denotes the mean curvature vector at $p$ of $S_i(p)$ viewed as a submanifold of the ambient space.
\item \  $\rank (S_i(p))=1$, where $S_i(p)$ is regarded as a full submanifold of its affine span.
\end{enumerate}
\end{lemma}

\begin{proof}
   It suffices to consider the case where $M$ is irreducible and compact. Therefore, we may assume that $M$ is irreducible and compact. If $M$ is isoparametric the result is clear. 
So let us assume that $M= K\cdot p$ is an orbit of an $s$-representation. In this case the normal holonomy representation at $p$ coincides with the slice representation of $K_p$. In particular, the set of fixed vectors of $K_{p|\nu _pM}$ coincides with $(\nu _0M)_p$. Let $\alpha$ and $\alpha^i$ be the second fundamental form of $M$ and $S_i(p)$, respectively. Let 
$e_1, \cdots , e_d$ be an orthonormal basis of $T_pS_i(p) = (E_i)_p$, where $d=\dim E_i$.
Then $$H_i(p)  = \sum\limits _{j=1}^d\alpha ^i (e_j, e_j) = \sum\limits _{j=1}^d\alpha (e_j, e_j).$$
If $k\in K_p$, then $$ H_i(p) = \sum\limits _{j=1}^d\alpha (e_j, e_j)= \sum\limits _{j=1}^d\alpha (ke_j, ke_j) = \sum\limits _{j=1}^d k\alpha (e_j, e_j) =kH_i(p).$$
Then $H_i(p)\in (\nu _0M)_p$. Now observe that for any $\xi \in \nu_0M$  (we have used that $K_p$ leaves invariant the eigenspace decomposition),
$$\langle \alpha (e_j,e_j), \xi\rangle  = \langle A_{\xi}e_j, e_j\rangle 
 = \langle (\eta _i)_p , \xi \rangle \langle e_j, e_j\rangle  = \langle (\eta _i)_p , \xi\rangle
.$$
Hence, $$\frac 1d H_i(p) = (\eta _i)_p .$$
This proves (1). 

Let $\xi$ be a parallel normal field to $M$ such that $\langle \eta _j , \xi\rangle =1$ if and only if $j=i$. Let us consider the  parallel focal manifold $M_\xi$ and the parallel endpoint map $\pi : M\to \mathbb R^{n+p}$, $\pi (x) = x +\xi (x)$. Then the $\ker \mathrm d \pi = E_i$ and the connected component containing $p$ of $\pi ^{-1}(\{\pi (p)\})$ coincides with $S_i(p)$. Since $\xi$ is $K$-invariant, $\pi$ is $K$-equivariant. This implies that 
the connected isotropy group $(K_{p+\xi(p)})^o$ coincides with the identity component of subgroup of $K$ that leaves invariant $S_i(p)$.  Then $S_i(p)$ is an orbit of the slice representation of $(K_{p+\xi(p)})^o$ on 
$\nu _{p +\xi (p)}\pi (M) = \nu _{p +\xi (p)}(K\cdot (p +\xi (p)))$. 
Note that  the image of this slice representation is, up to a trivial factor, an $s$-representation which coincides with the normal holonomy representation. Also observe that 
$K^o_p$ coincides with the connected isotropy at $p$ of $K_{p+\xi(p)}$, since the distribution $E_i$ is $K$-invariant. Namely,   
$$(((K_{p+\xi(p)})^o)_p)^o = ((K_{p+\xi(p)})_p)^o = K^o_p\, .$$
Putting together all this information, we obtain that  $(K_{p+\xi(p)})^o$ preserve the affine subspace $L$ generated by $S_i(p)$ and $(\nu _0S_i(p))_p\subset T_pL$ coincides with the set of fixed vectors of $(K_p)^o$ on the normal space of $S_i(p)\subset L$. We have used that the effective made action of $(K_{p+\xi (p)})^o$ on $L$ is an $s$-representation. Now observe that 
$(\nu _0S_i(p))_p\subset (\nu  M)_p$ and that $(\nu _0S_i(p))_p$ is pointwise fixed by $(K_p)^o$. 
Then $$(\nu _0S_i(p))_p\subset (\nu _0 M)_p \, .$$
If $\dim \,  (\nu _0S_i(p))_p\geq 2$, then there exists $\eta \in (\nu _0S_i(p))_p\subset (\nu _0 M)_p $ such that the shape operator $A_{\eta \vert T_pS_i(p)} $  of $M$ has at least two different eigenvalues. A contradiction that proves (2). 
\end{proof}

\medskip

Since $M$ is contained in a sphere, $(\eta_i)_p$ and $(\eta_j)_p$ are linearly independent  for all $p\in M$ whenever $i\neq j$. In fact, the radial vector field $\vec r$ is a parallel normal vector field and $\langle \vec r , \eta _i\rangle = 1 = \langle \vec r , \eta _j\rangle $. 
Let $i,j\in \{1,\ldots,g\}$, with $i\neq j$, and define $D_{i,j}$ as the subset of $\{1,\dots,g\}$ consisting of those $k$ such that $\eta_k$ lies in the parallel subbundle
\[
 \mathbb V_{i,j} :=\mathbb{R}\eta_i\oplus\mathbb{R}\eta_j
\]
of $\nu _0M$.
Observe that if $k,k'\in D_{i,j}$ with $k\neq k'$, then $D_{i,j}= D_{k,k'}$ and $\mathbb V_{i,j}= \mathbb V_{k,k'}$.

Let $i,j\in \{1, \dots , g\}$, with $i\neq j$ and let 
\begin{equation}\label{eq:Eij}
  E_{i,j} : = \bigoplus _{k\in D_{i,j}} E_k 
\end{equation}
The distribution $E_{i,j}$ is invariant under all the shape operators of $M$ (equivalently, 
$\alpha (E_{i,j}, E_{i,j}^\perp)= 0$).
Then $E_{i,j}$ is an autoparallel distribution of $M$. In fact,  

\begin{equation}\label{eq:autij}
    E_{i,j} = \bigcap _{\xi}\ker A_\xi
\end{equation}
where $\xi$ is a parallel section of $\mathbb V^\perp _{i,j}$. If $\zeta$
is any  generic parallel section of $\mathbb V^\perp _{i,j}$, then 
\begin{equation}\label{eq:autij2}
    E_{i,j} = \ker A_\zeta
\end{equation}
Let $M_{i,j}(p)$ be the (totally geodesic) integral manifold of $E_{i,j}$ that contains $p$. Observe that $M_{i,j}(p)$ is a submanifold with constant principal curvatures in the ambient space. Moreover,
\begin{equation}\label{eq:Mij}
M_{i,j}(p)\subset
p + (E_{i,j})_p \oplus \nu _pM,
\end{equation}
which need not be full in this affine space.
For consistency, we  define $D_{i,i}= \{i\}$ and  $\mathbb V_{i,i} = \mathbb R\eta_i$. In particular, $M_{i,i}(p)$ coincides with the (totally geodesic) integral manifold $S_i(p)$ of $E_i$ that contains $p$. 

One has that $\bar{\mathbb V}_{i,j}:=\mathbb V_{i,j|M_{i,j}(p)}$ is a parallel and flat subbundle of $\nu  _0M_{i,j}(p)$. Moreover, $\bar{\eta}_k: = \eta _{k|M_{i,j}(p)}$,  are the (parallel) curvature normals of $M_{i,j}(p)$ with respect to $\bar{\mathbb V}_{i,j}$, $k\in D_{i,j}$.
Note that $M_{i,j}(p) = M$ if and only if $\rank (M)=2$, $i\neq j$.

\begin{lemma}\label{lem:rank2} We keep the notation and assumptions of this section. If $|D_{i,j}|=2$, then $M_{i,j}(p) = S_i(p)\times S_j(p)$ for all $p\in M$, and $\eta_i\perp \eta_j$.
\end{lemma}
\begin{proof}
We have that $\bar E_i = E_{i|M_{i,j}(p)}$ and $\bar E_j = E_{j|M_{i,j}(p)}$ are mutually orthogonal and complementary autoparallel distributions; thus, they are parallel. Since $\alpha(\bar E_i,\bar E_j) = 0$, we obtain that $M_{i,j}(p)$ splits extrinsically as $M_{i,j}(p) = S_i(p) \times S_j(p)$ (we have used the so-called Moore's lemma \cite{Mo}; see also \cite[Lemma 1.7.1]{BCO}). By Lemma \ref{lem:398}(1), it follows that $\eta_i(p)\perp\eta_j(p)$. Since $\eta_i$ and $\eta_j$ are parallel, we conclude that $\eta_i\perp\eta_j$ everywhere.
\end{proof}

\begin{corollary}\label{cor:cpc}
Let $M\subset \mathbb{R}^N$ be a complete, full, and irreducible submanifold with constant principal curvatures and $\rank M\geq 2$, with associated curvature normals $\eta_1,\dots,\eta_g$, and eigendistributions $E_1, \dots , E_g$. Then for any $i\in \{1, \dots , g\}$ there exists $j\in\{1,\dots,g\}$, $i\neq j$, such that $|D_{i,j}|\geq 3$.
\end{corollary}
\begin{proof}
We may assume that $i=1$.
Suppose, to the contrary, that $|D_{1,j}|=2$ for all $j>1$. Then, by Lemma \ref{lem:rank2},
$\eta_1\perp\eta_j$ for all $j>1$.
Then $E_1$ and $\ker A_{\eta_1}=E_1^\perp=E_2\oplus\dots\oplus E_g$ are mutually orthogonal and complementary autoparallel distributions; thus, they are parallel. Since $\alpha(E_1,E_1^\perp)=0$, $M$ splits extrinsically by Moore's lemma. This is a contradiction.
\end{proof}

\begin{remark}\label{rem:A4} Keeping the notation and assumptions of this section, let $M$ be a Euclidean full and irreducible submanifold with constant principal curvatures and rank at least $2$. Let $i,j\in \{1, \dots , g\}$, $i\neq j$, and let $q\in M$. 
Then every curvature normal $(\eta _k)_q$, such that $k\in D_{i,j}$ lie in the line of the plane $(\mathbb V_{i,j})_q\subset \nu_0M$ given by the condition $\langle \vec r_q , \cdot \,\rangle =1$. This follows from the fact that $A_{\vec r}= Id$.
    
\end{remark}




\end{document}